\documentclass[11pt,oneside,english,reqno]{amsart}
\usepackage[T1]{fontenc}
\usepackage[latin9]{inputenc}
\usepackage[a4paper]{geometry}
\usepackage{mathrsfs}
\usepackage{amstext}
\usepackage{amsthm}
\usepackage{amssymb}
\usepackage{color}
\usepackage{hyperref}
\usepackage{dsfont}
\usepackage{enumerate}
\usepackage{verbatim}
\usepackage{enumitem}
\usepackage{cancel}

\makeatletter
\numberwithin{equation}{section}
\numberwithin{figure}{section}
\theoremstyle{plain}
\newtheorem{thm}{\protect\theoremname}
\theoremstyle{definition}
\newtheorem{defn}[thm]{\protect\definitionname}
\theoremstyle{remark}
\newtheorem{rem}[thm]{\protect\remarkname}
\theoremstyle{plain}
\newtheorem{lem}[thm]{\protect\lemmaname}
\theoremstyle{plain}
\newtheorem{prop}[thm]{\protect\propositionname}
\theoremstyle{plain}
\newtheorem{cor}[thm]{\protect\corollaryname}
\theoremstyle{plain}

\theoremstyle{plain}
\newtheorem{ex}[thm]{\protect\examplename}

\makeatother

\usepackage{babel}
\providecommand{\corollaryname}{Corollary}
\providecommand{\definitionname}{Definition}
\providecommand{\lemmaname}{Lemma}
\providecommand{\propositionname}{Proposition}
\providecommand{\remarkname}{Remark}
\providecommand{\theoremname}{Theorem}
\providecommand{\hypothesisname}{Hypothesis}
\providecommand{\examplename}{Example}

\newcommand{\la}{\langle}
\newcommand{\ra}{\rangle}
\newcommand{\cA}{\mathcal{A}}
\newcommand{\cB}{\mathcal{B}}
\newcommand{\cC}{\mathcal{C}}
\newcommand{\cD}{\mathcal{D}}
\newcommand{\cE}{\mathcal{E}}
\newcommand{\cF}{\mathcal{F}}

\newcommand{\cI}{\mathcal{I}}

\newcommand{\cL}{\mathcal{L}}

\newcommand{\cP}{\mathcal{P}}

\newcommand{\cS}{\mathcal{S}}

\newcommand{\scL}{\mathscr{L}}

\newcommand{\CC}{\mathbb{C}}

\newcommand{\EE}{\mathbb{E}}

\newcommand{\NN}{\mathbb{N}}

\newcommand{\PP}{\mathbb{P}}

\newcommand{\RR}{\mathbb{R}}

\newcommand{\fS}{\mathfrak{S}}

\newcommand{\dd}{\mathop{}\!\mathrm{d}}

\newcommand{\scB}{\mathscr{B}}

\title[Local times associated to Volterra-L\'evy processes and SDEs]{Regularity of Local times associated to Volterra-L\'evy processes and  path-wise regularization of stochastic differential equations}

\author{Fabian A. Harang \and Chengcheng Ling}

\begin{document}

\begin{abstract}
   We investigate the space-time regularity of the local time associated to Volterra-L\'evy processes, including Volterra processes driven by $\alpha$-stable processes for $\alpha\in(0,2]$. We show that the spatial regularity of the local time for Volterra-L\'evy process is $\PP$-a.s.  inverse proportionally to the singularity of the associated Volterra kernel. We apply our results to the investigation of path-wise regularizing effects obtained by perturbation of ODEs by a Volterra-L\'evy process which has sufficiently regular local time. Following along the lines of \cite{HarangPerkowski2020}, we show existence,  uniqueness and differentiablility of the flow associated to such equations.
\end{abstract}


\thanks{
\\
\emph{Acknowledgments}: F.A. Harang gratefully acknowledges financial support from the STORM project 274410, funded by the Research Council of Norway.
\\ Financial support for C. Ling by the DFG through the CRC 1283 "Taming uncertainty and profiting from randomness and low regularity in analysis, stochastics and their applications" is acknowledged.}

\address{Fabian A. Harang: email: fabianah@math.uio.no,
address:  Department of Mathematics, University of Oslo, P.O. box 1053, Blindern, 0316, OSLO, Norway}

\address{Chengcheng Ling: email: cling@math.uni-bielefeld.de, address:
Faculty of Mathematics, Bielefeld University,
33615, Bielefeld, Germany
}

\maketitle

 \tableofcontents{}

\section{Introduction}
Occupation measures and local times associated to $d$-dimensional paths $(p_t)_{t\in [0,T]}$ have received much attention  over the past decades from both in the  analytical  and  the probabilistic community. The occupation measure essentially quantifies the amount of time the path $p$ spends in a given set, i.e. for a Borel set $A\in \cB(\RR^d)$ the occupation measure is given by
\begin{equation*}
    \mu_t(A)=\lambda\{s\in [0,t]|\, p_s\in A\},
\end{equation*}
where $\lambda$ is the Lebesgue measure on $\mathbb{R}$.
The local time is given as the Radon-Nikodym derivative of the occupation measure with respect to the Lebesgue measure. The existence of the local time is generally not assured without some further knowledge of the path $p$, and the existence of the local time associated to the Weierstrass function, and other deterministic fractal like paths, is, to the best of our knowledge, still considered an open question. However, when $(p_t)_{t\in [0,T]}$ is a stochastic process, existence of the local time can often be proved using probabilistic techniques, and much research has been devoted to this aim, see e.g. \cite{GerHoro} and the references therein for a comprehensive overview.   Knowledge of probabilistic and analytic properties of the local time becomes useful in a variety problems arising in analysis. For example, given a measurable path $p$ with an existing local time, the following formula holds
\begin{equation*}
    \int_0^t b(x-p_s)\dd s=b\ast L_{t}(x),
\end{equation*}
where $\ast$ denotes convolution, and $L:[0,T]\times \RR^d \rightarrow \RR_+$ is the local time associated to $p$. Thus analytical or probabilistic questions relating to the left hand side integral can often be answered with the knowledge of the probabilistic and analytic properties of the local time $L$.
\\

In this article we will study regularity properties of the local time associated to Volterra-L\'evy processes given on the form
\begin{equation}\label{Volterra Levy}
z_t=\int_0^t k(t,s)\dd \scL_s, \qquad t\in [0,T],
\end{equation}
where $k(t,\cdot)\in L^\alpha([0,t])$ for all $t\in [0,T]$ with $\alpha\in(0,2]$, and $\scL$ is a L\'evy process on a filtered probability space $(\Omega,\cF,\PP)$. In the case when $\scL=B$ is a Brownian motion, then joint regularity in time and space of the local time associated to Volterra processes has received some attention in recent years as this knowledge can be applied towards regularization of ODEs by noise \cite{HarangPerkowski2020,galeati2020noiseless,galeati2020prevalence,Catellier2016}, as discussed in detail below.
Furthermore, in \cite{galeati2020prevalence}, the authors investigated the regularity of the local time associated to $\alpha$-stable processes, i.e. when the kernel $k\equiv 1$, and $\scL$ is an $\alpha$-stable process. One goal of this article is therefore to extend these results to the general case of Volterra-L\'evy processes, as well as apply this to the regularization by noise procedure.  Towards this end, we formulate a simple local non-determinism condition for these processes, which will be used to determine the regularity of the local time. The regularity of the local time is then proved in Sobolev space, by application of the recently developed stochastic sewing lemma \cite{le2018}, similarly as done for Gaussian Volterra
 processes in \cite{HarangPerkowski2020}. By  embedding, it follows that the local time is also contained in a wide range of Besov spaces.
 \\

 As an application of our results on regularity of the local time,  we show existence and pathwise uniqueness of SDEs of the form
 \begin{equation}\label{eq: genral intro equation}
      \frac{\dd }{\dd t}x_t = b(x_t)+\frac{\dd}{\dd t}z_t,\qquad x_0=\xi\in \RR^d
 \end{equation}
 even when $b$ is a Besov-distribution (the exact regularity requirement of $z$ and $b$ will be given in Section \eqref{sec. main results}  below).
It is well known that certain stochastic processes provide a regularizing effect on SDEs on the form of \eqref{eq: genral intro equation}. By this we mean that if the process $(z_t)_{t\in [0,T]}$ is given on some explicit form,   \eqref{eq: genral intro equation} might be well posed, even when $b$ does not satisfy the usual assumption of Lipschitz and linear growth.
In fact, in \cite{Catellier2016}, the authors show that if $z$ is given as a sample path of a fractional Brownian motion with Hurst index $H\in (0,1)$,
equation \eqref{eq: genral intro equation} is well posed and have a unique solution even when $b$ is only a distribution in the generalized Besov-H\"older space $\cC^\beta$ with $\beta<\frac{1}{2H}-2$.
More recently, Perkowski and one of the authors of the current article in \cite{HarangPerkowski2020} proved that there exists a certain class of continuous Gaussian processes with exceptional regularization properties. In particular, if $z$ in \eqref{eq: genral intro equation} is given as a path of such a process, then a unique solution exists to \eqref{eq: genral intro equation} (where the equation is understood in the pathwise sense), for any $b\in \cC^\beta$ with $\beta\in \RR$. Moreover, the flow map $\xi\mapsto x_t(\xi)$ is infinitely differentiable. We then say that the path $z$ is infinitely regularizing.  Not long after this result was published, Galeati and Gubinelli \cite{galeati2020noiseless},  showed that in fact {\em almost all continuous paths are infinitely regularizing} by using the concept of prevalence.
Furthermore, the regularity assumption on  $b$ was proven to be  inverse proportional to the irregularity of the continuous process $z$. In fact, this statement holds in a purely deterministic sense, see e.g \cite[Thm. 1]{galeati2020noiseless}.
The main ingredient in this approach to regularization by noise is to formulate the ODE/SDE into a non-linear Young equation, involving a non-linear Young integral, as was first described in \cite{Catellier2016}. This reformulation allows one to construct integrals, even in the case when traditional integrals (Riemann, Lebesgue, etc.) does not make sense. A particular advantage of this theory is furthermore that the framework itself does not rely on any probabilistic properties of the processes, such as Markov or martingale properties. This makes this framework particularly suitable when considering SDEs where the additive stochastic process is of a more exotic type. As is demonstrated in the current paper, the framework is well suited to study SDEs driven by   Volterra-L\'evy processes, which is a class of processes difficult to analyse using traditional probabilistic techniques. We believe that this powerful framework can furthermore be applied towards analysing several interesting problems relating to ill-posed SDEs and ODEs in the future.
\\

Historically, the investigation of similar regularising effects for SDEs with general L\'evy noise seems to have received less attention compared to the case when the SDE \eqref{eq: genral intro equation} is driven by a continuous Gaussian process. Of course, the general structure of the L\'evy noise excludes several techniques which has previously been applied in the Gaussian case.  However, much progress has been made also on this front when the equation has jump type noise, and although several interesting results deserves to be mentioned, we will only discuss here some the most recent results and refer the reader to \cite{Krylov2005,Flandoli20112,Bass2001,Flandoli2017,Zhang2017} for further results.  In \cite{priola2012}, Priola showed that \eqref{eq: genral intro equation} has a path-wise unique strong solution (in a probabilistic sense) when $z=\scL$ is a symmetric $\alpha$-stable process with $\alpha\in (0,2)$ and $b$ is a bounded $\beta$-H\"older continuous function of order $\beta>1-\frac{\alpha}{2}$.
In \cite{priola2018} this result was put in the context of path-by-path uniqueness suggested by Davie \cite{Davie2007}. More recently, in \cite{raynal2020weak} the authors prove that the martingale problem associated to \eqref{eq: genral intro equation} is well posed, even when $b$ is only assumed to be bounded and continuous, in the case when $z=\scL$ is an $\alpha$-stable process with $\alpha=1$ (being the critical case).  Further  in \cite{athreya2020}  the authors show strong existence and uniqueness of \eqref{eq: genral intro equation} when $z=\scL$ is an $1$-dimensional $\alpha$-stable process, and  $b\in \cC^\beta$ with $\beta>\frac{1}{2}-\frac{\alpha}{2}$. Thus, allowing here for possibly distributional coefficients $b$ when $\alpha$ is sufficiently large (i.e. greater than $1$).
Our results can be seen as an extension of the last result to a purely pathwise setting, and  to the case of general Volterra-L\'evy processes. Similarly as seen in the Gaussian case, the choice of Volterra kernel then dictates the regularity $\beta\in \RR$ of the  distribution $b\in \cC^\beta$ that can be considered to still obtain existence and uniqueness.

\subsection{Main results}\label{sec. main results}
We present here the main results to be proven in this article. The first result provides a simple condition to show regularity of the local time associated to Volterra-L\'evy processes.

\begin{thm}\label{thm: first main reg of local time}
Let $(\scL_t)_{t\in [0,T]}$ be a L\'evy process on a filtered probability space $(\Omega,\cF,\PP)$, with characteristic $\psi:\RR^d\rightarrow \CC$, and let $k$ be a real valued and possibly singular Volterra kernel satisfying for $t\in[0,T]$, $k(t,\cdot)\in L^\alpha([0,t])$ with $\alpha\in(0,2]$. Define the Volterra L\'evy process $(z_t)_{t\in [0,T]}$ by $z_t :=\int_0^t k(t,s)\dd \scL_s$,
where the integral is defined in  Definition \ref{VL}. Suppose that the characteristic triplet and the Volterra kernel satisfies for some $\zeta > 0$ and $\alpha\in(0,2]$
\begin{equation*}
    \inf_{t\in [0,T]}\inf_{s\in [0,t]} \inf_{\xi\in \RR^d} \frac{\int_s^t \psi(k(t,r)\xi)\dd r }{(t-s)^\zeta |\xi|^\alpha}>0.
\end{equation*}
If $\zeta\in (0,\frac{\alpha}{d})$, then there exists a $\gamma>\frac{1}{2}$ such that the local time $L:\Omega\times[0,T]\times \RR^d\rightarrow \RR_+$ associated to $z$ is contained in $\cC^\gamma([0,T]; H^\kappa(\RR^d))$ for any $\kappa<\frac{\alpha}{2\zeta}-\frac{d}{2}$, $\PP$-a.s..
\end{thm}

\begin{cor}\label{cor: test}
There exists a class of Volterra-L\'evy processes $z_t=\int_0^t k(t,s)\dd \scL_s$ such that for each $t\in [0,T]$, its associated local time $L_t$ is a test function. More precisely, we have that  $(t,x)\mapsto L_t(x)\in \cC^\gamma([0,T];\cD(\RR^d))$ $\PP$-a.s. for any $\gamma\in (0,1)$. Here $\cD(\RR^d)$ denotes the space of test functions on $\RR^d$.
\end{cor}
See Example \ref{exVL}, {\rm (iv)} for proof of this corollary.
\\

Inspired by \cite{HarangPerkowski2020,galeati2020prevalence,Catellier2016}
 we apply the result on regularity of the local time to prove regularization of SDEs by Volterra-L\'evy noise. Since we will allow the coefficient $b$ in \eqref{eq: genral intro equation} to be distributional-valued, it is not \emph{a priori} clear what we mean by a solution.  Indeed, since the integral $\int_0^t b(x_s)\dd s$ is not well defined in a Riemann or Lebesgue sense if $b$ is truly distributional, it is not a priori clear how to make sense of \eqref{eq: genral intro equation}. We therefore begin with the following definition of a solution, which is in line with  the definition of pathwise solutions to SDEs used in \cite{HarangPerkowski2020,galeati2020prevalence,Catellier2016}.

\begin{defn}\label{def: concept of solution}
Consider a Volterra-L\'evy process $z$ given as in \eqref{Volterra Levy} with measurable paths, and  associated local time $L$.  Let $b\in \mathcal{S}'(\mathbb{R}^d)$ be a distribution such that $b\ast L\in \cC^\gamma([0,T];\cC^2(\RR^d))$ for some $\gamma>\frac{1}{2}$. Then for any $\xi\in \RR^d$ we say that $x$ is a solution to
\begin{equation*}
    x_t=\xi+\int_0^t b(x_s)\dd s+z_t,\qquad \forall t\in [0,T],
\end{equation*}
if and only if $x-z\in \cC^\gamma([0,T];\RR^d)$, and there exists a $\theta\in \cC^\gamma([0,T];\RR^d)$ such that $\theta=x-z$, and $\theta $ solves the non-linear Young equation
\begin{equation*}
    \theta_t=\xi+\int_0^t b\ast \bar{L}_{\dd r} (\theta_r), \qquad \forall t\in [0,T].
\end{equation*}
Here $\bar{L}_t(z)=L_t(-z)$ where $L$ is the local time associated to $(z_t)_{t\in [0,T]}$, and the integral is interpreted in the non-linear Young sense, described in Lemma \ref{lem: non linear young integral}.
\end{defn}

\begin{thm}\label{thm: main existence and uniqueness}
Suppose $(z_t)_{t\in [0,T]}$ is a Volterra-L\'evy process such that its associated local time $L\in \cC^\gamma([0,T]; H^\kappa)$ for some $\kappa>0$ and $\gamma>\frac{1}{2}$, $\PP$-a.s.. Then for any  $b\in H^\beta(\RR^d)$ with $\beta>2-\kappa$, there exists a unique pathwise solution to the equation
\begin{equation*}\label{eq:SDE intro}
    x_t=\xi+\int_0^t b(x_s)\dd s+z_t,\qquad \forall t\in [0,T],
\end{equation*}
where the solution is interpreted in sense of Definition  \ref{def: concept of solution}. Moreover, if $\beta>n+1-\kappa$ for some $n\in \NN$, then the flow mapping $\xi\mapsto x_t(\xi)$ is $n$-times continuously differentiable.
\end{thm}

\subsection{Structure of the paper}
In section \ref{sec:occupation measures} we recall some basic aspects from the theory of occupation measures, local times, and Sobolev/Besov distribution spaces. Section \ref{sec:volterr levy process} introduces a class of Volterra processes where the driving noise is given as a L\'evy process. We show a construction of such processes, even in the case of singular Volterra kernels, and discuss conditions under which the process is continuous in probability. Several examples of Volterra-L\'evy processes are given, including a rough fractional $\alpha$-stable process, with $\alpha\in [1,2)$. In section \ref{sec: regualrity of local times} we provide some sufficient conditions for the characteristics of Volterra-L\'evy processes such that their associated local time exists, and is $\PP$-a.s. contained in a H\"older-Sobolev space of positive regularity. At last, we apply the concept of local times in order to prove regularization by noise for SDEs with additive Volterra-L\'evy processes. Here, we apply the framework of non-linear Young equations and integration, and thus our results can truly be seen as pathwise, in the "rough path" sense.
An appendix is included in the end, where statements and proofs of some auxiliary results are given.

\subsection{Notation}
For a fixed $T>0$, we will denote by $x_t$ the evaluation of a function at time $t\in [0,T]$, and write $x_{s,t}=x_t-x_s$. For some $n\in \NN$, we define
\begin{equation*}
    \Delta_T^n:=\{(s_1,\ldots,s_n)\in [0,T]^n|\, s_1\leq \dots \leq s_n\}.
\end{equation*}
To avoid confusion, the letter $\scL$ will be used to denote a L\'evy process, while $L$ will be used to denote the local time of a process.
For $\gamma\in (0,1)$ and a Banach space $E$, the space $\cC^\gamma_TE:=\cC^\gamma([0,T];E)$ is defined to be the space of functions $f:[0,T]\rightarrow E$ which is H\"older continuous of order $\gamma$. The space is equipped with the standard semi-norm
\begin{equation*}
    \|f\|_{\gamma}:=\sup_{s\neq t\in [0,T]}\frac{\|f_t-f_s\|_E}{|t-s|^\gamma},
\end{equation*}
and note that under the mapping $f\mapsto |f_0|+\|f\|_\gamma$ the space $\cC^\gamma_TE$ is a Banach space.

We let  $\cS(\RR^d)$ denote the Schwartz space of rapidly decreasing functions on $\RR^d$, and $\cS'(\RR^d)$ its dual space.
 Given $f\in\mathcal S(\mathbb{R}^d)$,
let $\mathscr{F} f$  be the Fourier transform of $f$ defined by
$$
\mathscr{F} f(\xi):=(2\pi)^{-d/2}\int_{\mathbb{R}^d}e^{-i\la\xi, x\ra} f(x)\dd x.
$$
Let $s$ be a real number. The Sobolev space $H^s(\mathbb{R}^d)$ consists of distributions $f\in \cS'(\RR^d)$ such that $\mathscr{F} f\in L_{loc}^2(\mathbb{R}^d)$ and
$$\Vert f\Vert_{H^s}^2:=\int_{\mathbb{R}^d}(1+|\xi|^2)^s|\mathscr{F} f(\xi)|^2\dd \xi<\infty.$$

For $\alpha>0$, if $\int_0^T|f(s)|^\alpha\dd s<\infty$, then we say $f\in L^\alpha([0,T])$.

\section{Occupation measures and local times, and distributions}\label{sec:occupation measures}
This section is devoted to give some background on the theory of occupation measures and local times, as well as definitions of Sobolev and  Besov spaces, which will play a central role throughout this article.

\subsection{Occupation measure and local times}
The occupation measure associated to a process $(x_t)_{t\in [0,T]}$ gives information about the amount of time the process  spends in a given set. Formally, we define the occupation measure $\mu$ associated to $(x_t)_{t\in [0,T]}$ evaluated at $t\in [0,T]$  by
\begin{equation*}
\mu_t(A)=\lambda\{s\leq t|x_s\in A\},
\end{equation*}
where $\lambda$ denotes the Lebesgue measure. The Local time $L$  associated to $x$ is then the Radon-Nikodym derivative with of $\mu$ with respect to the Lebesgue measure(as long as this exists). We therefore give the following definition.

\begin{defn}
Consider a process $x:[0,T]\rightarrow \RR^d$ be a process, and let $\mu$ denote the occupation measure of $x$. If there exists a function $L:[0,T]\times \RR^d \rightarrow \RR_+ $ such that
\begin{equation*}
\mu_t(A)=\int_A L_t(z)\dd z,\qquad {\rm for} \qquad A\in \cB(\RR^d),
\end{equation*}
then we say that $L$ is the local time associated to the process $(x_t)_{t\in [0,T]}$.
\end{defn}

\begin{rem}

The interpretation of the local time $L_t(z)$ is {the time spent by the process $x:[0,T]\rightarrow \RR^d$ at a given point $z\in \RR^d$ }. Thus, the study of this object has received much attention from people investigating both probabilistic and path-wise properties of stochastic processes. For purely deterministic processes $(x_t)_{t\in [0,T]}$, the local time might still exist, however, as discussed in \cite{HarangPerkowski2020}, if $x$ is a Lipschitz path, there exists at least two discontinuities of the mapping $z\mapsto L_t(z)$. On the other hand, it is well known (see \cite{GerHoro}) that the local time associated to the trajectory of a one dimensional Brownian motion is $\frac{1}{2}$-H\"older regular in its spatial variable ($a.s.$). More generally, for the trajectory of a fractional Brownian motion with Hurst index $H\in (0,1)$, we know that its local time $L$ is contained in $H^\kappa$ ($a.s.$) for $\kappa<\frac{1}{2H}-\frac{d}{2}$, while still preserving H\"older regularity in time. This clearly shows that the more irregular the trajectory of the fractional Brownian motion is, the more regularity we obtain in the local time associated to this trajectory. In this case, the regularity of the local time can therefore be seen as an irregularity condition. This heuristic has recently been formalized in \cite{galeati2020prevalence}. There,  the authors show that if the local time associated to a continuous path $(x_t)_{t\in [0,T]}$ is regular (i.e. H\"older continuous or better) in space, then $x$ is \emph{truly rough}, in the sense of \cite{Friz2014}.
More recently, the authors of \cite{HarangPerkowski2020} showed that the local time associated to trajectories of certain particularly irregular Gaussian processes (for example the log-Brownian motion) is infinitely differentiable in space, and {almost} Lipschitz in time.  In the current article, we will extend this analysis to L\'evy processes.
\end{rem}

The next proposition will be particularly interesting towards  applications in differential equations, and which we will use in subsequent sections.

\begin{prop}[{\rm Local time formula}] Let $b$ be a measurable function, and suppose $(x_t)_{t\in [0,T]}$ is a process with associated local time $L$. Then the following formula holds for any $\xi \in \RR^d$ and  $(s,t)\in \Delta^2_T$
\begin{equation*}
\int _s^t b(\xi+x_r)\dd r=b\ast \bar{L}_{s,t}(\xi),
\end{equation*}
where $\bar{L}_t(z)=L_t(-z)$ and  $L_{s,t}=L_t-L_s$ denotes the increment.
\end{prop}
A proof of this statement follows directly from the definition of the local time, see \cite[Thm. 6.4]{GerHoro} for further details.

\begin{rem}
It is readily seen that, formally, the local time can be expressed in the following way for $\xi\in \RR^d$ and $(s,t)\in \Delta^2_T$
\begin{equation*}
L_{s,t}(\xi)=\int_s^t \delta(\xi - X_r)\dd r,
\end{equation*}
where $\delta$ is the Dirac distribution.
\end{rem}
\begin{rem}
For future reference, we also recall here that the Dirac distribution $\delta$ is contained in the in-homogeneous Sobolev space $H^{-\frac{d}{2}-\epsilon}$ for any $\epsilon>0$ (See e.g. \cite[Remark 1.54]{Bahouri2011}).
\end{rem}

\subsection{Besov spaces and distributions}
Before introducing the notion of Besov spaces, we give a definition of the Paley-Littlewood blocks, which plays a central role in the construction of these spaces.
\begin{defn}[Paley-Littlewood blocks]
 For  $j\in \NN$,  $\rho_j:=\rho(2^{-j}\cdot)$ where $\rho$ is a smooth function supported on an annulus $\cA:=\{x\in\mathbb{R}^d:\frac{4}{3}\leq |x|\leq \frac{8}{3}\}$ and $\rho_{-1}$ is a smooth function supported on the ball $B_{\frac{4}{3}}$. Then $\{\rho_j\}_{j\geq -1}$ is a partition of unity (\cite{Bahouri2011}).
For $j\geq -1$  and some $f\in \cS^\prime$ we define the Paley-Littlewood blocks $\Delta_j$ in the following way
\begin{equation*}
    \Delta_j f=\mathscr{F}^{-1}(\rho_j\mathscr{F}{f}).
\end{equation*}

\end{defn}

\begin{defn}
For $\alpha\in \mathbb{R}$  and $p,q\in[1,\infty]$, the in-homogeneous  Besov space $\cB^\alpha_{p,q}$  is defined by
\begin{equation*}
    \cB_{p,q}^\alpha =\Big\{f\in \mathcal{S}^\prime \Big|\,\,\|f\|_{B_{p,q}^\alpha}:= \left(\sum _{j\geq-1} 2^{jq\alpha}\|\Delta_j f\|_{L^p(\mathbb{R}^d)}^q\right)^{\frac{1}{q}} <\infty \Big\}.
\end{equation*}
We will typically write $\cC^\alpha:=\cB_{\infty,\infty}^\alpha$. Besides, by the definition of the partition of unity and Fourier-Plancherel formula (\cite[Examples p99]{Bahouri2011}), the Besov space $\cB^\alpha_{2,2}$ coincides with Sobolev space $H^\alpha$.
\end{defn}

\begin{rem}
We will work with regularity of the local time in the Sobolev space $H^\kappa$. However, towards applications to regularization by noise in SDEs, we will also encounter Besov spaces, through Young's convolution inequality. We therefore give a definition of these spaces here. Of course, through Besov embedding, $H^\kappa \hookrightarrow B^{\kappa-(\frac{d}{2}-\frac{d}{p})}_{p,q}$ for any $p,q\in [2,\infty]$ and $\kappa \in \RR$, (e.g. \cite[Prop. 2.20]{Bahouri2011}), and thus our results implies that the local time is also included in these Besov spaces. We will however not specifically work in this setting to avoid extra confusion, but refer the reader to \cite{galeati2020noiseless,galeati2020prevalence}  for a good overview of regularity of the local time associated to Gaussian processes in such spaces.
\end{rem}

\section{Volterra-L\'evy process}\label{sec:volterr levy process}

In this section we give a brief introduction on L\'evy processes and stochastic integral for a Volterra kernel with respect to a  L\'evy process. General references for this part are \cite[Chp. 4]{Sato1990} and \cite[Chp. 2, Chp. 4]{Applebaum2004}. In Section 3.1 we give the definition of Volterra-L\'evy    processes (with possibly singular kernels) and obtain the associated characteristic function. Particularly, our framework include Volterra processes driven by symmetric $\alpha$-stable noise.  In the end, we provide several examples of  Volterra-L\'evy processes, including  the \emph{fractional $\alpha$-stable process}.

We begin to provide a definition of L\'evy processes, as well as a short discussion on a few important properties.

\begin{defn}[L\'evy process]\label{alphastable}
Let $T>0$ be fixed.  We say that a c\`adl\`ag and ($\mathcal{F}_t$)-adapted stochastic process $(\scL_t)_{t\in [0,T]}$ defined on a complete probability space $(\Omega,\mathcal{F},(\mathcal{F}_t)_{t\in [0,T]},\PP)$, and which satisfies the usual assumptions is a {\em L\'evy process} if the following properties hold:
\begin{itemize}[leftmargin=.3in]
    \item[{\rm (i)}] $\scL_0=0$ ($\PP$-a.s.).
    \item[{\rm (ii)}] $\scL$ has independent and stationary increments.
    \item[{\rm (iii)}] $\scL$ is  continuous in probability, i.e. for all $\epsilon>0$, and all $s>0$,
    $$\lim_{t\rightarrow s}\PP(|\scL_t-\scL_s|>\epsilon)=0.$$
\end{itemize}
Furthermore, let $\nu$ be a $\sigma$-finite measure on $\mathbb{R}^d$. We say that it is a \emph{L\'evy measure} if
\begin{align*}
  \nu(\{0\})=0,\quad \int_{\mathbb{R}^d}(1\wedge|x|^2)\nu(\dd x)<\infty.
\end{align*}
\end{defn}

\begin{rem}
 A known description of L\'evy process is L\'evy-Khintchine formula: for a $d$-dimensional L\'evy process $\scL$, the characteristic function $\psi$ of $\scL$ verifies that for  $t\geq0$, there exists a vector $a\in\mathbb{R}^d$,  a positive definite symmetric $d\times d$ matrix $\sigma$ and   a L\'evy measure $\nu$
such that the characteristic function is given by  $\EE[e^{i\la\xi,\scL_t\ra}]=e^{-t\psi(\xi)}$ with
\begin{align}\label{LK}
   \psi(\xi)=-i\la a,\xi\ra+\frac{1}{2}\la \xi,\sigma\xi\ra-\int_{\mathbb{R}^d-\{0\}}(e^{i\la\xi,x\ra}-1-i\la\xi,x\ra1_{|x|\leq1}(x))\nu(\dd x).
\end{align}
Here  the triple $(a,\sigma,\nu)$ is called  the \emph{characteristic} of the random variable $\scL_1$.
\end{rem}

The typical examples for L\'evy processes is the case when the L\'evy triplet is given by $(0,\sigma,0)$, resulting in a Brownian motion. Another typical example is when the characteristic triplet is given by $(0,0,\nu)$ and the L\'evy measure $\nu$ defines an $\alpha$-stable process. We provide the following definition for this class of processes.
\begin{defn}[Standard $\alpha$-stable process]
If a $d$-dimensional L\'evy process $(\scL_t)_{t\geq0}$
has the following characteristic function
\begin{align*}
    \psi(\xi)=c_\alpha|\xi|^\alpha,\quad \xi\in\mathbb{R}^d
\end{align*}
with $\alpha\in(0,2]$ and some positive constant $c_\alpha$, then we say $(\scL_t)_{t\geq0}$
is a  standard $\alpha$-stable process.
\end{defn}

We now move on to the construction of Volterra-L\'evy processes, given of the form
\begin{equation}\label{eq:volt processes}
    z_t=\int_0^t k(t,s)\dd \scL_s,\qquad t\in [0,T].
\end{equation}
Of course in the case when $(\scL_t)_{t\in [0,T]}$ is a Gaussian process, or even a square integrable martingale, the construction of such a stochastic integral is by now standard, and  $z$ is constructed as an element in  $L^2(\Omega)$ given that $k(t,\cdot)\in L^2([0,t])$ for all $t\in [0,T]$, see e.g. \cite{Protter2004}.  However, in the case when $\scL$ is not square integrable, then the construction of $z$ as a stochastic integral is not as straight forward. However, several articles discuss also this construction in the case of $\alpha$-stable processes, which would be sufficient for our purpose. The next remark gives only a brief overview on this construction, and we therefore  ask the interested reader to consult the given references for further details on the construction.

\begin{rem}\label{stableint}
Consider  a symmetric $\alpha$-stable process $\scL$ with $\alpha\in(0,2)$. From \cite[Ex. 25.10, p162]{Sato1990} we know that $\EE[|\scL_t|^p]=Ct^{p/\alpha}$ for any $-1<p<\alpha$ and $t\in[0,T]$, and thus the process is not square integrable and the standard "It\^o type" construction of the Volterra process in \ref{eq:volt processes} can not be applied.
However,  in \cite[Chp. 3.2-3.12]{ST1994} the authors propose several different ways of constructing integral  $\int_0^tk(t,s)\dd\scL_s$ given that $k(t,\cdot)\in L^\alpha([0,t])$. In particular, in   \cite[Chp. 3.6]{ST1994} it is shown that the Volterra-stable process below is well-defined and exists in $L^p(\Omega)$ for any $p<\alpha$, given that the kernel $k(t,\cdot)\in L^\alpha([0,t])$ for all $t\in [0,T]$. In fact, in the case when $\scL$ is a symmetric $\alpha$-stable process, it is known that for any $0<p<\alpha$
\begin{equation*}
\left(\EE\left[\left|\int_0^t k(t,s)\dd\scL_s\right|^p\right] \right)^{\frac{1}{p}}\simeq_{p,\alpha,d} \left(\int_0^t |k(t,s)|^\alpha\dd s\right)^{\frac{1}{\alpha}},
\end{equation*}
where $\simeq_{p,\alpha,d}$ means that they differ up to a constant depending on $p,\alpha$ and $d$ (recall that $d$ is the dimension of $\scL$).  See e.g. \cite{ROSINSKI1986} and the references therein for more details on this relation and the  construction of such integrals.
\end{rem}

The above discussion yields the following definition of the Volterra-L\'evy process.

\begin{defn}[Volterra-L\'evy process]\label{VL}
Fix $T>0$, and let $(\scL_t)_{t\in [0,T]}$ be a L\'evy process as given in Definition  \ref{alphastable}. For a given kernel $k:\Delta_T^2\rightarrow\mathbb{R}$ with the property that for any $t\in [0,T]$, $k(t,\cdot)\in L^\beta([0,t])$ with $\beta\in(0,2]$, define
$$
z_t=\int_0^tk(t,s)\dd \scL_s,\quad t\geq0
$$
where the integral is constructed in  $ L^p(\Omega)$ sense for $p\leq \beta$, as discussed above.  Then we call the stochastic process $(z_t)_{t\in [0,T]}$  a {\em  Volterra-L\'evy process}, where $\scL$ is the associated L\'evy process to $z$ and $k$ is called the Volterra kernel.
\end{defn}

\begin{prop}\label{integral}
Let $(\scL_t)_{t\in [0,T]}$ be a L\'evy process on a probability space $(\Omega,\cF,\PP)$, such that $\EE[|\scL_t|^p]<\infty$ for all $0<p<\beta$ where $\beta\in (0,2]$. If  $k(t,\cdot)\in L^{\beta}([0,t])$  for any $t\in [0,T]$, then the Volterra \-L\'evy process  $(z_t)_{t\in [0,T]}$ given by
\begin{equation*}
    z_t=\int_0^t k(t,s)\dd \scL_s
\end{equation*}
is well defined as an element of $L^{p}(\Omega)$ for any $0<p<\beta$.  For $0\leq s\leq t\leq T$, the characteristic function of $z$ is given by
\begin{align}\label{chfz}
\EE[\exp(i \la \xi, z_t\ra)]=\exp \left(-\int_0^t\psi(k(t,s)\xi ) \dd s\right),
\end{align}
and the conditional characteristic function is given by
\begin{align}\label{chfcz}
\EE[\exp(i \la \xi, z_t\ra)|\cF_s]=\cE_{0,s,t}(\xi)\exp \left(-\int_s^t\psi(k(t,r)\xi )\dd r\right),
\end{align}
where $\cE_{0,s,t}(\xi):=\exp\left(i\la \xi, \int_0^sk(t,r)\dd \scL_r \ra \right)$.
\end{prop}
\begin{proof}
The fact that $z_t\in L^p(\Omega)$ for any $0<p<\beta$ follows from Remark \ref{stableint}. However, the statement is even stronger in the case when $\scL$ is a square integrable martingale, as in this case it is well known that if $k(t,\cdot)\in L^2([0,t])$ for any $t\in [0,T]$ then $z_t\in L^2(\Omega)$ for any $t\in [0,T]$.
By application of a standard  \emph{dominated convergence theorem} argument, it is readily checked that the characteristic function satisfies the following relations
\begin{align*}
\EE[\exp(i \la \xi, z_t\ra)]&= \EE\Big[\lim_{n\rightarrow\infty}\exp\Big(\sum_{j=1}^n i \la k(t,s_j)\xi,(\scL_{s_{j+1}}-\scL_{s_j})\ra\Big)\Big]\\
&=\lim_{n\rightarrow\infty} \prod_{j=0}^{n}\EE\Big[\exp\Big(i \la k(t,s_j)\xi,(\scL_{s_{j+1}}-\scL_{s_j})\ra\Big)\Big]\\
&=\lim_{n\rightarrow\infty} \exp \Big(-\sum_{j=0}^{n}(s_{j+1}-s_j)\psi(k(t,s_j)\xi)\Big)
\\
&=\exp \left(-\int_0^t\psi(k(t,s)\xi ) \dd s\right).
\end{align*}
It follows that \eqref{chfz} holds. For \eqref{chfcz}, since $\int_s^tk(t,r)\dd\scL_r$ is independent of $\cF_s$ and $\int_0^sk(t,r)\dd \scL_r$ is adapted to  $\cF_s$ for any $s\in[0,t]$, we similarly have
\begin{align*}
\EE[\exp(i \la \xi, z_t\ra)|\cF_s]&=\EE[\exp(i \la \xi, \int_s^tk(t,r)\dd\scL_r+\int_0^sk(t,r)\dd\scL_r\ra)|\cF_s]
\\&=\cE_{0,s,t}(\xi)\exp \left(-\int_s^t\psi(k(t,r)\xi )\dd r\right).
\end{align*}
\end{proof}

Everything we have  introduced so far only relates to the probabilistic properties of Volterra-L\'evy process without any details regarding its  sample path behavior. Towards the goal of proving regularity of the local time associated to $(z_t)_{t\in [0,T]}$, as done in  Section \ref{sec: regualrity of local times}, we require that the process $z$ is continuous in probability.  We therefore provide here a simple sufficient condition on the kernel $k$ that assures this property of the process.

\begin{lem}\label{lem:cont in prob}
Let $(z_t)_{t\in [0,T]}$ be a Volterra L\'evy process, as given in Definition \ref{VL}. Then $z$ is continuous in probability if there exists a $p>0$ such that
\begin{equation*}
    \EE[|z_t-z_s|^p]\rightarrow 0 \quad {\rm when} \quad s\rightarrow t.
\end{equation*}
\end{lem}
\begin{proof}
 It is readily checked that for any $p>0$ and any $\epsilon>0$
 \begin{equation*}
     \PP(|z_t-z_s|>\epsilon)\leq \frac{1}{\epsilon^p} \EE[|z_t-z_s|^p],
 \end{equation*}
 and thus if $\EE[|z_t-z_s|^p]\rightarrow 0$ as  $s\rightarrow t$, continuity in probability holds.
\end{proof}

Below we provide three examples of different types of Volterra processes driven by L\'evy noise.

\begin{ex}[Brownian motion]
 Let $\beta=2$, $k(t,\cdot)\in L^2([0,t])$ for $t\in[0,T]$. Suppose $\scL$ is a Brownian motion with values in $\mathbb{R}^d$. Then it is well known  that $z_t=\int_0^tk(t,s)\dd \scL_s$ is well-defined in $L^2(\Omega)$ as a Wiener integral.
 The sample paths of such processes are clearly measurable, and   depending on the regularity of the kernel $k$, the process may also be (H\"older) continuous.
\end{ex}

\begin{ex}[Square-integrable martingale case]
 Let $\beta=2$,  $k(t,\cdot)\in L^2([0,t])$ for $t\in[0,T]$ and $\scL$ be a $(\mathcal{F}_t)$-martingale satisfying $\EE[|\scL_t|^2]<\infty$, for all $t\in[0,T]$. Then we know  $z_t=\int_0^tk(t,s)\dd \scL_s,t\geq0$ is well-defined according to Proposition \ref{integral} (this is also clear from classical martingale theory, e.g. \cite{Applebaum2004}).

\end{ex}

\begin{ex}[Standard $\alpha$-stable case]\label{ex:standard alpha stable}
  For $\alpha\in(0,2)$, {$\beta>-\frac{1}{\alpha}$}, such that $|k(t,s)|\simeq |t-s|^\beta$ and  $k(t,\cdot)\in L^\alpha([0,T])$, let $\scL$ be a standard $\alpha$-stable process defined in Definition \ref{alphastable}.  By Proposition \ref{integral}, we know that $z_t=\int_0^tk(t,s)\scL_s,t\geq0$ is well defined for $\alpha\in(0,2)$ (see also  \cite[Section 3.6 Examples]{ST1994}).
Furthermore, $(z_t)_{t\in [0,T]}$ is continuous in probability. Indeed, note that $z_t-z_s=\int_s^t k(t,r)\dd\scL_r+\int_0^s k(t,r)-k(s,r)\dd\scL_r$. Let us consider the first integral in this decomposition, as the second follows similarly.  By Remark \ref{stableint} it follows that there exists a $p>0$ such that
\begin{equation}
    \EE[|\int_s^t  k(t,r)\dd \cL_s|^p]\simeq_{p,\alpha,d} \left(\int_s^t k(t,r)^\alpha \dd r\right)^{\frac{p}{\alpha}}.
\end{equation}
Since $|k(t,r)|\simeq |t-r|^\beta$, then
\begin{equation}
   H(s,t):= \int_s^t k(t,r)^\alpha\dd r\simeq \int_s^{t} |t-r|^{\beta\alpha} \dd r=(\beta\alpha+1)^{-1}|t-s|^{\beta\alpha+1}.
\end{equation}
Since $\alpha>0$ and $\beta> -1/\alpha$, it follows that $H(s,t)\rightarrow 0$ as $s\rightarrow t$. Treating similarly the term $\int_0^s k(t,r)-k(s,r)\dd \scL_r$, and using that
\begin{equation}
\EE[|z_t-z_s|^p]\lesssim_p \EE[|\int_s^t k(t,r)\dd\scL_r|^p]+\EE[|\int_0^s k(t,r)-k(s,r)\dd\scL_r|^p],
\end{equation}
we conclude by Lemma \ref{lem:cont in prob} that $(z_t)_{t\in [0,T]}$ is continuous in probability.
\end{ex}

With the above preparation at hand, we can then construct fractional $\alpha$-stable processes and give a representation of its characteristic function. We summarize this in the following example.

\begin{ex}[\rm{Fractional $\alpha$-stable process}]\label{fr}
Let  $\scL$ be an $\alpha$-stable process with $ \alpha\in(0,2]$, and consider the Volterra kernel  $k(t,s)=(t-s)^{H-\frac{1}{\alpha}}$,
 $H\in(0,1)$. Then the process $z_t=\int_0^tk(t,s)\dd\scL_s$   is called a \emph{fractional $\alpha$-stable process} (of Riemann-Liouville type) and specifically if $\alpha=2$, then $\scL$ is a Brownian motion and $z$ is a  \emph{fractional Brownian motion}.  Note that in this case $k(t,\cdot)\in L^\alpha([0,t],\dd s)$ for any $H\in (0,1)$. There is a more detailed study of fractional processes of this type in \cite[Chapter 7]{ST1994}.
An application of Proposition  \eqref{chfz} yields that the characteristic function associated to the fractional $\alpha$-stable process $z$, is given by
$$\EE[\exp(i \la \xi, z_t\ra)]=\exp \left(-c_\alpha \frac{|\xi|^\alpha t^{H\alpha}}{H\alpha} \right).$$
Note also that by the same argument as used in Example \ref{ex:standard alpha stable}, it is readily checked that the fractional $\alpha$-stable process $(z_t)_{t\in [0,T]}$ is continuous in probability.
\end{ex}

\section{Regularity of the  local time associated to Volterra-L\'evy processes}\label{sec: regualrity of local times}
This section is devoted to prove space-time regularity of the local time associated to Volterra-L\'evy processes, as defined in Section \ref{sec:volterr levy process}.
 We begin to give a notion of local non-determinism for these processes, and provide a few examples of specific processes which satisfy this property.

\subsection{Local non-determinism condition for Volterra-L\'evy process}

The following definition of a local non-determinism condition can be seen as an extension of the concept of strong local non-determinism used in the context of Gaussian processes, see e.g. \cite{Xiao2006,galeati2020prevalence,HarangPerkowski2020}.

\begin{defn}\label{def: zeta alpha Volterra kernel}
Let $\scL$ be a L\'evy process with characteristic $\psi:\RR^d\rightarrow \CC$ as given in \eqref{LK}, and
let $z$ be a   Volterra-L\'evy process (Definition \ref{VL}) with  L\'evy process $\scL$ and  Volterra kernel $k:\Delta^2_T\rightarrow \RR$  satisfying $k(t,\cdot)\in L^\alpha([0,t])$ for all $t\in [0,T]$.  If for some  $\zeta>0$ and   $\alpha\in(0,2]$ the following inequality holds
\begin{equation}\label{LND}
\lim_{t\downarrow 0} \inf_{s\in(0,t]} \inf_{\xi\in \RR^d} \frac{\int_s^t \psi(k(t,r)\xi) \dd r}{(t-s)^\zeta|\xi|^\alpha}>0.
\end{equation}
Then we say that  $z$ is  $(\alpha,\zeta)$-Locally non-deterministic ($(\alpha,\zeta)$-LND).
\end{defn}

 \begin{rem}
  The elementary example of a Volterra kernel is $k(t,s):=1_{[0,t]}(s)$ for any $0\leq s\leq t\leq T$. In this case the  Volterra-L\'evy process is just given as the L\'evy process itself, i.e.  $z_t =\scL_t$.  If we let $\scL$ be a standard $d$-dimensional $\alpha$-stable process, condition \eqref{LND} fulfills for $\zeta=1$. Hence  a standard $d$-dimensional $\alpha$-stable process is $(\alpha,1)$-LND, which coincides with the conclusion in \cite[Proposition 4.5, Example (a)]{Nolan89}.
 \end{rem}

\begin{rem}

There already exists several concepts of local non-determinism, but, as far as we know, most of them are given in terms of a condition on the variance of certain stochastic processes. The only exception we are aware of is the definition of Nolan in \cite{Nolan89} for $\alpha$-stable processes, where a similar condition is stated in $L^p$ spaces, with $p=\alpha$ (see \cite[Definition 3.3]{Nolan89}).  Of course working with general $\alpha$-stable processes, we do in general not have finite variance, and thus the standard definitions of such a concept is not applicable. On the other hand, in the case when $\alpha=2$, we have finite variance, and then the above criterion would be very similar to the condition for strong local non-determinism for Gaussian Volterra processes, as discussed for example in \cite{Xiao2006}. Working with  the conditional characteristic function of  Volterra $\alpha$-stable  processes, we see however that this condition in some sense is what needs to be replaced in order to prove existence and regularity of local times associated to these processes.

\end{rem}

 It is readily seen that the Volterra $\alpha$-stable process satisfies \eqref{LND}, with $\zeta$  depending on the choice of kernel $k$. The condition is however some what more general, as we only require the processes to behave similarly to  Volterra  $\alpha$-stable processes. Let us provide an example to discuss some interesting process that satisfies the LND condition.

\begin{ex}[Volterra kernel] \label{VK} As two examples of Volterra kernel that we are interested most, we give a specific discussion here. The first one usually relates to fractional type processes, for instance, fractional Brownian motion and fractional stable processes. As we will see later,  the second one makes the  corresponding Volterra-L\'evy process  an infinitely regularising process, similarly to the Gaussian counterpart discussed in \cite{HarangPerkowski2020}.
\begin{itemize}[leftmargin=.3in]
    \item[\rm (i)]For $\alpha\in(0,2]$, $H\in(0,1)$, let $k(t,s)=F(t,s)(t-s)^{H-\frac{\alpha}{2}}$, where $F:\Delta_T^2\rightarrow \RR\setminus{0}$ is continuous and $F(t,s)\simeq 1$ when $|t-s|\rightarrow 0$, where $\simeq$ means that the two sides are comparable up to a positive constant. It can be easily checked that $k(t,\cdot)\in L^\alpha([0,t])$ for $t\in[0,T]$.
    \item[\rm (ii)]  Let $p>\frac{1}{\alpha}$, and consider the kernel
    $k(t):=t^{-\frac{1}{\alpha}}(\ln \frac{1}{t})^{- p}$ for $t\in [0,1)$.
    It is readily seen that $k(t,\cdot)\in L^\alpha([0,t])$ for any $t<1$.
\end{itemize}

\end{ex}

\begin{ex}[Gaussian case: $\alpha=2$]
\label{Gaussian}
 Let $\scL=B$ be a Brownian motion. Then the Gaussian Volterra process $z_t=\int_0^tk(t,s)\dd B_s, t\in[0,T]$ is $(2,\zeta)$-LND according to definition \ref{def: zeta alpha Volterra kernel} if
 \begin{align*}
    \lim_{t\downarrow 0} \inf_{s\in(0,t]} \frac{\int_s^t |k(t,r)|^2 \dd r}{(t-s)^\zeta}>0.
\end{align*}

\end{ex}
As we mentioned, the L\'evy process $\scL$ does not have to be Gaussian type processes. For non-Gaussian type $\scL$ we mostly consider $\alpha$-stable processes or the processes which has similar behavior to stable processes. Since the condition \eqref{LND} only focuses on the characteristic function $\psi$ of $\scL$, there is a large class of jump processes which can be studied here.
\begin{ex}[Stable type processes]\label{Non-Gaussian}
Fix an $\alpha\in(0,2)$.  Given a kernel $k:\Delta^2_T\rightarrow \RR$  with $k(t,\cdot)\in L^\alpha([0,t])$ and satisfying for some  $\zeta>0$ the following inequality
\begin{align}
    \label{k}
    \lim_{t\downarrow 0} \inf_{s\in(0,t]} \frac{\int_s^t |k(t,r)|^\alpha \dd r}{(t-s)^\zeta}>0.
\end{align}
Then the following list of processes satisfy the LND condition in Definition \ref{def: zeta alpha Volterra kernel}:
\begin{itemize}[leftmargin=.3in]
    \item[\rm(i)]  $\scL$ is a standard $d$-dimensional $\alpha$-stable process, i.e., $$\psi(\xi)=c_\alpha|\xi|^\alpha,\quad c_\alpha>0.$$ Then
obviously $z_t =\int_0^t k(t,r)\dd\scL_r, t>0$ is $(\alpha,\zeta)$-LND.\\
Besides, here if $k(t,s)=k(t-s)$ for $0\leq s\leq t<\infty$ and $\int_0^t|k(t,s)|^\alpha \dd s>0$, according to Definition \ref{def: zeta alpha Volterra kernel}, the process $z$ is $(\alpha,1)$-LND, which coincides the conclusion  in \cite[Proposition 4.5]{Nolan89}.
\item[\rm(ii)] $\scL=(\scL_1,\cdots,\scL_d)$, where $\scL_1,\cdots,\scL_d$ are independent $1$-dimensional standard $\alpha$-stable processes.
In this case the corresponding characteristic function $\psi$ is given by
$$
\psi(\xi)=c_\alpha(|\xi_1|^\alpha+\cdots+|\xi_d|^\alpha),\quad c_\alpha>0.
$$
By Jensen's inequality, it follows that $|\xi_1|^\alpha+...+|\xi_d|^\alpha
=|\xi_1|^{2\cdot \frac{\alpha}{2}}+...+|\xi_d|^{2\cdot \frac{\alpha}{2}}
\geq (|\xi_1|^2+...+|\xi_d|^2)^{\frac{\alpha}{2}}=|\xi|^\alpha$ for $\alpha\in(0,2]$, which  implies  $\psi(\xi)\geq c_\alpha|\xi|^\alpha$. By \eqref{k} we conclude that $z_t =\int_0^t k(t,r)\dd\scL_r, t>0$, is $(\alpha,\zeta)$-LND.
\item[\rm(iii)] $\scL$ is a $d$-dimensional L\'evy process with characteristic function $$\psi(\xi)=|\xi|^{\alpha}{\log(2+|\xi|)},\quad\xi\in \mathbb{R}^d.$$ We additionally assume $\alpha\in(0,1)$ (see \cite{Kang2015} Example 1.5). This processes is not really a stable process but the small size jumps  of this process has similar behavior to stable processes.
Since $|\xi|^{\alpha}{\log(2+|\xi|)}\geq |\xi|^\alpha$ for $\xi\in \mathbb{R}^d$, then
\begin{align*}
    \lim_{t\downarrow 0} \inf_{s\in(0,t]} \inf_{\xi\in \RR^d} \frac{\int_s^t \psi(k(t,r)\xi) \dd r}{(t-s)^\zeta|\xi|^\alpha}&\geq  \lim_{t\downarrow 0} \inf_{s\in(0,t]} \inf_{\xi\in \RR^d} \frac{\int_s^t |k(t,r)|^{{\alpha}}\dd r}{(t-s)^\zeta}>0.
\end{align*}
Therefore $z_t =\int_0^t k(t,r)\dd\scL_r, t>0$, is $(\alpha,\zeta)$-LND.
\end{itemize}
\end{ex}

The following theorem shows the regularity of the local time associated to  Volterra-L\'evy  processes which is $(\alpha,\zeta)$-LND according to Definition \ref{def: zeta alpha Volterra kernel}.

\subsection{Regularity of the local time}
With the concept of local non-determinism at hand, we are now ready to prove the regularity of the local time associated to Volterra L\'evy processes, and thus also proving Theorem \ref{thm: first main reg of local time}.
The following theorem provides a proof of Theorem \ref{thm: first main reg of local time}, as well as giving $\PP$-a.s. bounds for the Fourier transform of the occupation measure and the local time.

\begin{thm}[{\rm Regularity of Local time}]\label{thm: regualrity of local times associated to alpha Volterra process}
 Let  $z:\Omega\times [0,T]\rightarrow \RR^d$ be a L\'evy Volterra process with characteristic $\psi:\RR^d \rightarrow \CC$ on a complete filtered probability space $(\Omega,\cF,\{\cF_t\}_{t\in [0,T]},\PP)$, and suppose $z$ is $(\alpha,\zeta)$-LND for some $\zeta\in(0,\frac{\alpha}{d})$  and $\alpha\in(0,2]$,  continuous in probability, and adapted to the filtration $(\cF_t)_{t\in [0,t]}$.  Then the local time $L:\Omega\times [0,T]\times \RR^d \rightarrow \RR_+$ associated to $z$ exists and is square integrable. Furthermore,  for any $\kappa<\frac{\alpha}{2\zeta}-\frac{d}{2}$ there exists a  $\gamma>\frac{1}{2}$ such that  the local time is contained in the space $\cC^\gamma_T H^\kappa$.
\end{thm}
\begin{proof}
We will follow along the lines of the proof of \cite[Theorem 17]{HarangPerkowski2020}, but adapt to the case of L\'evy processes. To this end, we will apply the stochastic sewing lemma from \cite{le2018}, which is  provided in Lemma \ref{Lem: Stochastic sewing lemma} for self-containedness.

A Fourier transform of the occupation measure $\mu_{s,t}(\dd x)$ yields  $\int_s^t e^{i\la\xi, z_r\ra}\dd r$.  Note that this coincides with the Fourier transform of the local time $L_{s,t}(x)$ whenever $L$ exists.
Our first goal is therefore to show that for any $p\geq 2$, the following inequality holds for some $\lambda\geq 0$ and $\gamma \in (\frac{1}{2},1)$
\begin{equation*}
    \|\widehat{\mu_{s,t}}(\xi)\|_{L^p(\Omega)}\lesssim (1+|\xi|^2)^{-\frac{\lambda}{2}}|t-s|^\gamma.
\end{equation*}
To this end, the stochastic sewing lemma (see Lemma \ref{Lem: Stochastic sewing lemma}) will provide us with this information. We begin to define
$$A_{s,t}^\xi:=\int_s^t \EE[\exp(i\la \xi, z_r\ra)|\cF_s]\dd r,$$
 and for a partition $\cP[s,t]$ of $[s,t]$ define
\begin{equation*}
    \cA_{\cP[s,t]}^\xi:=\sum_{u,v} A_{u,v}^\xi
\end{equation*}
If the integrand $A^\xi$ satisfy the  conditions {\rm (i)-(ii)} in Lemma \ref{Lem: Stochastic sewing lemma}, then a unique limit to $\cA_{s,t}^\xi=\lim_{|\cP|\rightarrow 0}\cA_{\cP}^\xi$ exists in $L^p(\Omega)$.  Note that then $\int_s^t e^{i\la \xi, z_r\ra} \dd r =\cA_{s,t}^\xi$ in $L^p(\Omega)$. We continue to prove that conditions {\rm (i)-(ii)} in Lemma \ref{Lem: Stochastic sewing lemma} is indeed satisfied for our integrand $A$. It is already clear that $A^\xi_{s,s}=0$, and $A^\xi_{s,t}$ is ($\cF_t$)-measurable.
For any point $u\in [s,t]$  we define
$$
\delta_u f_{s,t}:=f_{s,t}-f_{s,u}-f_{u,t}
$$
for any function $f:[0,T]^2 \rightarrow \RR$.
It follows by the tower property  and linearity of conditional expectations that
\begin{multline*}
    \EE[\delta_u A^\xi_{s,t}|\cF_s]=\EE[\int_s^t \EE[\exp(i\la \xi,z_r\ra)|\cF_s]\dd r
    \\
    -\int_s^t \EE[\exp(i\la \xi, z_r\ra)|\cF_s]\dd r-\int_s^u \EE[\exp(i\la \xi, z_r\ra)|\cF_u]\dd r |\cF_s]=0.
\end{multline*}
At last, we will need to control the term $\|\delta_u A_{s,t}^\xi\|_{L^p(\Omega)}$.  To this end, using Proposition \ref{integral}, we know that
\begin{equation*}
    A_{s,t}^\xi =  \int_s^t \cE_{0,s,r}(\xi) \exp\left(- \int_s^r \psi(k(r,l)\xi) \dd l \right) \dd r,
\end{equation*}
where $\cE$ is defined as in \eqref{chfcz}.
Therefore, it is readily checked that
\begin{equation*}
    \delta_{u}A_{s,t}^\xi = \int_u^t \cE_{0,s,r}(\xi) \exp\left(-\int_s^r \psi(k(r,l)\xi) \dd l \right)
    - \cE_{0,u,r}(\xi) \exp\left(-\int_u^r \psi(k(r,l)\xi) \dd l \right) \dd r.
\end{equation*}
Of course, moments of the complex exponential $\cE_{0,s,r}(\xi)$ is bounded by $1$, i.e.  for any $r\in[s,t]$, $\|\cE_{0,s,r}(\xi)\|_{L^p(\Omega)}\leq 1$, and therefore it follows that
\begin{equation*}
    \|\delta_u A_{s,t}^\xi\|_{L^p(\Omega)} \lesssim \int_u^t \exp\left(-\int_s^r \psi(k(r,l)\xi) \dd l \right)+\exp\left(-\int_u^r \psi(k(r,l)\xi) \dd l\right) \dd r.
\end{equation*}
Using the fact that $z$ is $(\alpha,\zeta)$-LND for some $\zeta\in (0,1)$, and using that $(r-s)^\zeta\geq (r-u)^\zeta $ we obtain the estimate
\begin{equation*}
     \|\delta_u A_{s,t}^\xi\|_{L^p(\Omega)}  \lesssim \int_u^t \exp\left(-c |\xi|^\alpha (r-u)^\zeta \right) \dd r.
\end{equation*}
Note in particular that this holds for any $p\geq 2$.
By the property of the exponential function, we have that for any $\eta\in \RR_+$
\begin{equation}\label{eq: exponential bound}
    \exp\left(-c_\alpha |\xi|^\alpha (r-u)^\zeta \right)\leq \exp(T^\zeta) \left(1+|\xi|^\alpha\right)^{-\eta}(r-u)^{-\zeta\eta}\sup_{q\in \RR_+} q^\eta \exp(-q).
\end{equation}

Since $1+|\xi|^\alpha\lesssim (1+|\xi|^2)^\frac{\alpha}{2}$ for all $\alpha\in(0,2]$, applying this relation in
 \eqref{eq: exponential bound}, and assuming that $0\leq \eta\zeta<\frac{1}{2}$ it follows that for all $p\geq 2$ there exists a $\gamma>\frac{1}{2}$ and $\lambda<\frac{\alpha}{2\zeta}$
\begin{equation*}
     \|\delta_u A_{s,t}^\xi\|_{L^p(\Omega)} \lesssim (1+|\xi|^2)^{-\frac{\lambda}{2}}(t-u)^{\gamma}.
\end{equation*}
Thus, both conditions of \eqref{eq:integrand cond} in  Lemma \ref{Lem: Stochastic sewing lemma} are satisfied. By a simple addition and subtraction of the integrand $A_{s,t}^\xi$ in  \eqref{bounds on stochastic integral}, it follows that for all $p\geq 2$ the limiting process $\cA_{s,t}^\xi$ satisfies
\begin{equation}\label{eq:bound on integral A}
    \|\cA_{s,t}^\xi\|_{L^p(\Omega)} \lesssim (1+|\xi|^2)^{-\frac{\lambda}{2}}(t-u)^{\gamma}.
\end{equation}
We will now show that $\widehat{\mu_{s,t}}(\xi)=\cA_{s,t}^\xi$ in $L^p(\Omega)$.  For a partition $\cP$ of $[s,t]$, we have
\begin{align*}
    \|\widehat{\mu_{s,t}}(\xi)-\cA^\xi_{s,t}\|_{L^p(\Omega)} \leq \sum_{[u,v]\in \cP} \int_u^v \|e^{i\la \xi,z_r\ra}-\EE[e^{i\la \xi,z_r\ra}|\cF_u]\|_{L^p(\Omega)}\dd r.
\end{align*}
 By Minkowski's inequality, we have
\begin{equation*}
    \|e^{i\langle \xi,z_r\rangle} -\EE[e^{i\langle \xi,z_r\rangle}|\cF_u]\|_{L^p(\Omega)}\leq \|e^{i\langle \xi,z_r\rangle} -e^{i\langle \xi,z_u\rangle}\|_{L^p(\Omega)}+\|\EE[e^{i\langle \xi,z_r\rangle}-e^{i\langle \xi,z_u\rangle}|\cF_u]\|_{L^p(\Omega)},
\end{equation*}
and by Jensen's inequality it follows  that
\begin{equation*}
    \|\EE[e^{i\langle \xi,z_r\rangle}-e^{i\langle \xi,z_u\rangle}|\cF_u]\|_{L^p(\Omega)}\leq \|e^{i\langle \xi,z_r\rangle}-e^{i\langle \xi,z_u\rangle}\|_{L^p(\Omega)}.
\end{equation*}
This implies  that
\begin{equation*}
      \|e^{i\langle \xi,z_r\rangle} -\EE[e^{i\langle \xi,z_r\rangle}|\cF_u]\|_{L^p}\leq 2  \|e^{i\langle \xi,z_r\rangle} -e^{i\langle \xi,z_u\rangle}|\cF_u\|_{L^p(\Omega)}.
\end{equation*}
Furthermore, for fixed $\epsilon>0$ then
\begin{align*}
    \|e^{i\langle \xi,z_r\rangle}-e^{i\langle \xi,z_u\rangle}\|_{L^p(\Omega)}&
    \leq \|(e^{i\langle \xi,z_r\rangle}-e^{i\langle \xi,z_u\rangle})1_{|e^{i\langle \xi,z_r\rangle}-e^{i\langle \xi,z_u\rangle}|>\epsilon}\|_{L^p(\Omega)}+\epsilon
    \\
    &\leq \PP(|e^{i\langle \xi,z_r\rangle}-e^{i\langle \xi,z_u\rangle}|>\epsilon)+\epsilon
\end{align*}
since $\|e^{i\langle \xi,z_r\rangle}-e^{i\langle \xi,z_u\rangle}\|_{L^q(\Omega)}\leq 1$ for any $q$.
We conclude that for any $\epsilon>0$
\begin{equation}\label{eq:compare}
    \|\widehat{\mu_{s,t}}(\xi)-\cA^\xi_{s,t}\|_{L^p(\Omega)} \leq 2\epsilon(t-s)+   2\sum_{[u,v]\in \cP} \int_u^v\PP(|e^{i\langle \xi,z_r\rangle}-e^{i\langle \xi,z_u\rangle}|>\epsilon)\dd r.
\end{equation}
Since this holds for any partition $\cP$, letting the mesh tend to $0$, {and using the assumption that $z$ is continuous in probability, it follows that $e^{i\langle \xi,z_r\rangle}$ is continuous in probability (see e.g. the continuous mapping theorem), and thus
\begin{equation*}
   \lim_{|\cP|\rightarrow 0} \sum_{[u,v]\in \cP} \int_u^v\PP(|e^{i\langle \xi,z_r\rangle}-e^{i\langle \xi,z_u\rangle}|>\epsilon)\dd r \leq \lim_{|\cP|\rightarrow 0} \sup_{[u,v]\in \cP} \sup_{r\in [u,v]} \PP(|e^{i\langle \xi,z_r\rangle}-e^{i\langle \xi,z_u\rangle}|>\epsilon)(t-s)=0.
\end{equation*}
Since \eqref{eq:compare} holds for any partition (also for partitions with infinitesimal mesh), we conclude that for any $\epsilon>0$
\begin{equation*}
     \|\widehat{\mu_{s,t}}(\xi)-\cA^\xi_{s,t}\|_{L^p(\Omega)}\leq 2\epsilon(t-s)
\end{equation*}
and} since $\epsilon$ can be chosen arbitrarily small,  we conclude that $\widehat{\mu_{s,t}}(\xi)=\cA^\xi_{s,t}$ in $L^p(\Omega)$.

We move on to estimate the Sobolev norm $\|L_{s,t}\|_{H^{\kappa}}$ for some appropriate $\kappa\in \RR$. We begin to observe that
\begin{equation*}
    \|\|\mu_{s,t}\|_{H^\kappa}\|_{L^p(\Omega)}=\left[\EE\left(\int_{\RR^d} (1+|\xi|^2)^{\kappa}|\widehat{\mu_{s,t}}(\xi)|^2 d\xi\right)^\frac{p}{2}\right]^\frac{1}{p}.
\end{equation*}
By Minkowski's inequality, it follows that
\begin{equation*}
    \|\|\mu_{s,t}\|_{H^\kappa}\|_{L^p(\Omega)} \lesssim \|(1+|\cdot|^2)^{\frac{\kappa}{2}}\|\widehat{\mu_{s,t}}\|_{L^p(\Omega)}\|_{L^2(\RR^d)},
\end{equation*}
and then use the bound from \eqref{eq:bound on integral A} to observe that
\begin{equation*}
    \|\|\mu_{s,t}\|_{H^\kappa}\|_{L^p(\Omega)} \lesssim (t-s)^{\gamma}\|(1+|\cdot|^2)^{\frac{\kappa-\lambda}{2}}\|_{L^2(\RR)}.
\end{equation*}
Choosing $\kappa=\lambda-\frac{d}{2}-\epsilon$ for some arbitrarily small $\epsilon>0$, it follows that

 $$
 \|(1+|\cdot|^2)^{\frac{\kappa-\lambda}{2}}\|_{L^2(\RR)}= \|(1+|\cdot|^2)^{\frac{d}{4}-\epsilon/2}\|_{L^2(\RR)}<\infty.
 $$
Recalling that $\lambda<\frac{\alpha}{2\zeta}$, since $\epsilon>0$ could be chosen arbitrarily small, we obtain that for any $\kappa<\frac{\alpha}{2\zeta}-\frac{d}{2}$ there exists a $\gamma>\frac{1}{2}$ such that
\begin{equation*}
    \|\|\mu_{s,t}\|_{H^\kappa}\|_{L^p(\Omega)} \lesssim (t-s)^{\gamma}.
\end{equation*}
Since $p\geq 2$ can be chosen arbitrarily large,  we conclude by Kolmogorov's theorem of continuity that
 there exists a set $\Omega^\prime \subset \Omega$ of full measure  such that for all $\omega\in \Omega^\prime$ there exists a $C(\omega)>0$ such that
\begin{equation*}
    \|\mu_{s,t}(\omega)\|_{H^\kappa}\leq C(\omega)(t-s)^{\gamma}
\end{equation*}
In particular, this implies that for almost all $\omega\in \Omega$,  $\mu(\omega)\in L^2(\RR^d)$ and thus the local time $L(\omega)$ (given as the density of $\mu$) exists, and our claim follows.
\end{proof}

\begin{rem}
In the case when $\alpha=2$, then $X$ is a Gaussian process, and Theorem \ref{thm: regualrity of local times associated to alpha Volterra process} provides the same regularity of the associated local time for a Gaussian Volterra process as proven in for example \cite{HarangPerkowski2020} (or without considering the joint time regularity, as shown in  \cite{GerHoro,Berman73, Pitt78}). This theorem can therefore be seen as an extension of this work to the class of Volterra-L\'evy processes.
\end{rem}
We will now give several examples on the application of Theorem \ref{thm: regualrity of local times associated to alpha Volterra process} to show the regularity of the local time for a few specific Volterra $\alpha$-stable processes. All of the following examples also were studied for dimension $d=1$ in \cite[Corollary 4.6, Examples]{Nolan89}, it shows that the local time $L_{t}(x)$ of  a Volterra $\alpha$-stable process exists $a.s.$ and is continuous for $(t,x)\in[0,T]\times\mathbb{R}$, furthermore for fixed $t\in[0,T]$, $L_{t}(x)$ is H\"older continuous for $x\in \mathbb{R}$ with some order less than $1$.   The method therein \cite{Nolan89} heavily relies on the $L^\alpha$-representation for $\alpha$-stable processes.
\begin{ex}[\rm{Regularity of the local time for Volterra $\alpha$-stable processes}]\label{exVL}
 We consider  Volterra-$\alpha$-stable processes
 $$z_t=\int_0^tk(t-s)\dd \scL_s,\quad t\geq0,$$
 where $\scL$ is a $d$-dimensional standard $\alpha$-stable process with $\alpha\in(0,2]$.
\begin{itemize}[leftmargin=.3in]
    \item[\rm (i).] Let $d=1$ and $k(t)\equiv 1$ for all $t\geq0$, then $$z_t=\scL_t$$ is an $1$-dimensional standard $\alpha$-stable process. When $\alpha\in(1,2]$, we know that an $1$-dimensional standard $\alpha$-stable process $\scL$ is $(\alpha,1)$-LND, and  continuous in probability. According to above theorem there exists a $\gamma>\frac{1}{2}$, such that the local time associated to $z$, and thus also $\scL$  is contained in $\cC^\gamma_T H^\kappa$ for any $\kappa<\frac{\alpha}{2}-\frac{1}{2}$, $\PP$-a.s..
    \item[(ii)] Let $k(t)=e^{-at}$ with $a>0$. Then the Ornstein-Uhlenbeck L\'evy process
    $$z_t=\int_0^t e^{-a(t-s)}\dd \scL_s,\quad t\geq0$$
    is $(\alpha,\zeta)$-LND for $\alpha\in(0,2]$ and  $\zeta=1$, and continuous in probability. Hence  there exists a $\gamma>\frac{1}{2}$, such that  the process $z$ has a local time $L\in\cC^\gamma_T H^\kappa$ for any $\kappa<\frac{\alpha}{2}-\frac{d}{2}$,  $\PP$-a.s..

    \item[\rm (iii)] Let $(z_t)_{t\in [0,T]}$ be a fractional $\alpha$-stable process as in Example \ref{fr}. Then $(z_t)_{t\in [0,T]}$ is continuous in probability, and  there exists a $\gamma>\frac{1}{2}$ such that the local time $L$ associated to $z$ is contained in $\cC^\gamma_TH^\kappa$ for any $\kappa<\frac{1}{2H}-\frac{d}{2}$, $\PP$-a.s.. Note that in this case, one obtains the same regularity for the local time, as one would for the fractional Brownian motion (see e.g. \cite{HarangPerkowski2020}).

    \item[\rm (iv)] Fix $T<1$ and let $k(t)=t^{-\frac{1}{\alpha}}\ln(\frac{1}{t})^{-p}$ for some $p>1$, and suppose $(\scL_t)_{t\in [0,T]}$ is a standard $\alpha$-stable process for some $\alpha\in (0,2]$. Let $(z_t)_{t\in [0,T]}$ be the Volterra L\'evy process built from $k$ and $\scL$. It is readily checked with Lemma \ref{lem:cont in prob}, Example \ref{ex:standard alpha stable} and Example \ref{VK} that this process is continuous in probability. Moreover,  note that in this case $z$ is $(\alpha,\zeta)$-LND for any $\zeta>0$.
    Thus for any $\gamma\in (0,1)$  the local time $L$ associated to $z$ is contained in $\cC^\gamma_T H^\kappa$ for any $\kappa\in \RR$, $\PP$-a.s.. Furthermore, if $z$ is c\`adl\`ag, then the local time $L$ has compact support, $\PP$-a.s. and thus $L\in \cC^\gamma([0,T];\cD(\RR^d))$, $\PP$-a.s. where $\cD(\RR^d)$ denotes the space of test functions on $\RR^d$. This proves in particular Corollary \ref{cor: test}.

\end{itemize}

\end{ex}

\section{Regularization of ODEs perturbed by  Volterra-L\'evy  processes}\label{sec:regualrization by noise}
With the knowledge of the spatio-temporal regularity of the local time associated to a Volterra-L\'evy process, we can solve additive SDEs with possibly distributional-valued drift's coefficients. The goal of this section is to prove Theorem \ref{thm: main existence and uniqueness}. To this end, we will recall some of the tools from the theory of non-linear Young integrals and corresponding equations. This theory for construction of integrals and equations is by now well known (see e.g. \cite{Catellier2016, hu2017nonlinear} and more recently \cite{HarangPerkowski2020,galeati2020noiseless} for an overview), but for the sake of self-containedness, we have included some short versions of proofs in the appendix. We also  mention that conditions for existence and uniqueness of non-linear Young equations can be stated in more general terms than what is used here. We choose to work with a simple set of conditions to provide a clear picture of the regularising effect in SDEs driven by Volterra-L\'evy noise, in contrast to the full generality which could be accessible. More general conditions for existence and uniqueness of non-linear Young equations can for example be found in \cite{galeati2020noiseless}.

\begin{lem}\label{lem: abstract young equations}
Suppose $\Gamma:[0,T]\times \RR^d \rightarrow \RR^d$ is contained in $\cC^{\gamma}_T\cC^\kappa$  for some $\gamma\in (\frac{1}{2},1)$ and $\kappa> \frac{1}{\gamma}$, and satisfies the following inequalities for $(s,t)\in \Delta^T_2$ and $\xi,\tilde{\xi}\in \RR^d$
\begin{equation}\label{eq: conditions for ex and uni}
\begin{aligned}
{\rm (i)}&\qquad |\Gamma_{s,t}(\xi)|+|\nabla \Gamma_{s,t}(\xi)| \lesssim |t-s|^\gamma
\\
{\rm (ii)} &\qquad  |\Gamma_{s,t}(\xi)-\Gamma_{s,t}(\tilde{\xi})| \lesssim |t-s|^\gamma |\xi-\tilde{\xi}|
\\
{\rm (iii)}& \qquad |\nabla \Gamma_{s,t}(\xi)-\nabla \Gamma_{s,t}(\tilde{\xi})|\lesssim |t-s|^\gamma |\xi-\tilde{\xi}|^{\kappa-1}.
\end{aligned}
\end{equation}
Then for any $\xi\in \RR^d$ there exists a unique solution to the equation
\begin{equation}\label{eq: general ODE}
y_t=\xi+\int_0^t \Gamma_{\dd r}(y_r).
\end{equation}
Here the integral is interpreted as the non-linear Young integral described in Appendix \ref{app: nonlinear young equations}
\begin{equation*}
\int_0^t \Gamma_{\dd r}(y_r)=\lim_{|\cP|\rightarrow 0} \sum_{[u,v]\in \cP} \Gamma_{u,v}(y_u),
\end{equation*}
for any partition $\cP$ of $[0,t]$.
\end{lem}
\begin{proof}
See proof in Appendix \ref{app: nonlinear young equations}.
\end{proof}

From here on, all analysis is done pathwise. That is, we now consider a subset $\Omega'\subset \Omega$ of full measure such that for all $\omega\in \Omega'$ the local time $L(\omega)$ associated to a Volterra-L\'evy process is contained in $\cC_T^\gamma H^\kappa$, for $\gamma$ and $\kappa$ as given through Theorem \ref{thm: regualrity of local times associated to alpha Volterra process}. With a slight abuse of notation, we will write $L=L(\omega)$. \\

Before moving on to prove existence and uniqueness of ODEs perturbed by  Volterra L\'evy processes, we will need  a technical proposition on the convolution of the local time with certain (possibly distributional) vector fields.

\begin{prop}\label{prop: regularity of conv}
Let $(z_t)_{t\in [0,T]}$ be a Volterra L\'evy process which is continuous in probability and  $(\alpha,\zeta)$-LND    for some $\zeta\in (0,1]$, such that  the associated local time $L$ is contained in $\cC^\gamma_T H^\kappa$ for some $\gamma>\frac{1}{2}$ and $\kappa <\frac{\alpha}{2\zeta}-\frac{d}{2}$. Suppose $b\in H^\beta$ for some $\beta\in \RR$. Then the following inequality holds for any $\theta<\beta+\kappa$ and $(s,t)\in \Delta^2_T$
\begin{equation}\label{reg of conv}
\|b\ast \bar{L}_{s,t}\|_{\cC^{\theta}}\lesssim \|b\|_{H^\beta} \|L\|_{\cC^\gamma_TH^\kappa}|t-s|^\gamma, \qquad \PP-a.s.
\end{equation}
Here, $\bar{L}_t(x)=L_t(-x)$.
\end{prop}
\begin{proof}
From Theorem  \ref{thm: regualrity of local times associated to alpha Volterra process}, we know that $L_{s,t}\in H^\kappa$ for $\kappa<\frac{\alpha}{2\zeta}-\frac{d}{2}$, thus, an application of Young's convolution inequality, reveals that \eqref{reg of conv} holds.
\end{proof}

A combination of Lemma \ref{lem: abstract young equations} and Proposition \ref{prop: regularity of conv} provides the existence and uniqueness of ODEs perturbed by $(\alpha,\zeta)$-LND  Volterra L\'evy  processes.
The following corollary and proposition can be seen as proof of Theorem \ref{thm: main existence and uniqueness}.

\begin{cor}[SDEs driven by stable Volterra processes]\label{cor: ex and uni sde}
Let $(z_t)_{t\in [0,T]}$ be a Volterra-L\'evy process which is continuous in probability, and  $(\alpha,\zeta)$-LND  according to definition \ref{def: zeta alpha Volterra kernel} for some $\zeta\in (0,\frac{\alpha}{d})$ and $\alpha\in (0,2]$. Suppose $b\in H^\beta$ for some $\beta \in \RR$ such that the following inequality holds $\beta+\frac{\alpha}{2\zeta}-\frac{d}{2}\geq 2$.
Then for any $\xi\in \RR^d$ there exists a unique solution $y\in \cC^\gamma_T(\RR^d)$ to the equation
\begin{equation}\label{eq:particular equation}
    y_t=\xi+\int_0^t b\ast \bar{L}_{\dd r}(y_r),\qquad t\in [0,T].
\end{equation}
Here the integral and solution is interpreted pathwise in sense of Lemma \ref{lem: abstract young equations} by setting $\Gamma_{s,t}(x):=b\ast \bar{L}_{s,t}(x)$, where we recall that $\bar{L}_t(x)=L_t(-x)$ and $L$ is the local time associated to $ (z_t)_{t\in [0,T]}$.
\end{cor}
\begin{proof}
By Proposition \ref{prop: regularity of conv}, we know that $b\ast L \in \cC^\gamma_T\cC^{\theta}$ for any $\theta<\beta+\frac{\alpha}{2\zeta}-\frac{d}{2}$. Since $\beta+\frac{\alpha}{2\zeta}-\frac{d}{2}\geq 2$, set $\Gamma_{s,t}(x):=b\ast \bar{L}_{s,t}(x)$, and it follows directly that conditions {\rm (i)-(iii)} of Lemma \ref{lem: abstract young equations} are satisfied, and thus a unique solution to    \eqref{eq:particular equation} exists.
\end{proof}

Additionally to existence and uniqueness, the authors of \cite{HarangPerkowski2020} provided a general program to prove higher order differentiability of the flow mapping $\xi \mapsto y_t(\xi)$. We will here apply this program in order to show differentiability of flows associated to ODEs perturbed by
 sample paths of a Volterra-L\'evy process. It is well known that if $b\in C^k$ for some $k\geq 1$, the the flow $\xi\mapsto y_t(\xi)$ where $y$ is the solution to the ODE
\begin{equation*}
    y_t =\xi+\int_0^t b(y_r)\dd r,
\end{equation*}
is $k$-times differentiable. Translating this to the abstract framework of non-linear Young equations; let $y$ be the solution to
\begin{equation*}
    y_t=\xi+\int_0^t \Gamma_{\dd r}(y_r),\qquad \xi\in \RR^d,
\end{equation*}
where $\Gamma\in \cC^\gamma_T\cC^\kappa$ for some $\kappa>\frac{1}{\gamma}$. Then the flow $\xi\mapsto y_t(\xi)$ is $\kappa$ times differentiable. Recall that $\Gamma$ in our setting represents the convolution between the (possibly distributional) vector field $ b$ and the local time associated to the irregular path of a Volterra-L\'evy process. We therefore provide a proposition to highlight the relationship between the regularity of the vector field $b$, the regularity of the local time associated to a Volterra L\'evy process, and the differentiability of the flow.

\begin{prop}
Let $(z_t)_{t\in [0,T]}$ be a Volterra-L\'evy process taking values in $\RR^d$ which is continuous in probability, and  $(\alpha,\zeta)$-LND  for some $\alpha\in (0,2]$ and $\zeta\in(0,\frac{\alpha}{d})$. Suppose $b\in H^\beta$ for some $\beta\in \RR$ such that $\beta+\frac{\alpha}{2\zeta}-\frac{d}{2}\geq 1+n$ for some integer $n\geq 1$. Let $y(\xi)\in \cC^\gamma_T(\RR^d)$ denote the solution to \eqref{eq:particular equation} starting in $\xi\in \RR^d$. Then, the flow map $\xi\mapsto y_\cdot(\xi)$  is $n$-times Fr\'echet differentiable.
\end{prop}

\begin{proof}
This result for abstract Young equations was proven in \cite[Thm. 2]{HarangPerkowski2020}, but we give a short outline of the proof here.
Denote by $\theta=\beta+\frac{\alpha}{2\zeta}-\frac{d}{2}$, and since $\theta\geq 2$ it follows that there exists a unique solution to \eqref{eq:particular equation}.
We will prove the differentiability of the flow by induction, and begin to show the existence of the first derivative. It is readily checked that the first derivative of the flow $\xi\mapsto y_\cdot(\xi)$ needs to satisfy the equation
\begin{equation}\label{eq:gradient linear ODE}
    \nabla y_t(\xi)=1 +\int_0^t \nabla \Gamma_{\dd r}(y_r(\xi)) \nabla y_r(\xi), \qquad {\rm for} \qquad  t\in [0,T],
\end{equation}
where the integral is understood in sense of the non-linear Young integral in Lemma \ref{lem: non linear young integral},by setting  $\Gamma^1_{s,t}(u_s)=\nabla \Gamma_{s,t}(y_s(\xi)) u_s $ where  $u_s=\nabla y_s(\xi) \in \RR^{d\times d}$.  Since \eqref{eq:gradient linear ODE} is a linear equation, existence and uniqueness can be simply verified following along the lines of the proof of Lemma \ref{lem: abstract young equations}.

\end{proof}

\appendix

\section{Stochastic Sewing Lemma}
We recite the stochastic Sewing lemma given in \cite{le2018} for self containedness. However, we refer to the aforementioned article for a discussion and full proof of this statement.

\begin{lem}\label{Lem: Stochastic sewing lemma}
Let $(\Omega,\mathcal{F},\{\mathcal{F}_t\}_{t\in[0,T]},\mathbb{P})$ be a complete probability space, where $\mathcal{F}_0$ contains all $\mathbb{P}$-null sets. Suppose $p\geq 2$ and let $A:\Delta^2_T\rightarrow \mathbb{R}^d$ be a stochastic process such that $A_{s,s}=0$,  $A_{s,t}$ is ($\mathcal{F}_t$)-measurable, and $(s,t)\mapsto A_{s,t}$ is right-continuous from  $\Delta^2_T$ into $L^p(\Omega)$. Set $\delta_u A_{s,t}:=A_{s,t}-A_{s,u}-A_{u,t}$ for all $(s,u,t)\in \Delta^3_T$, and assume that there exists constants $\beta>1$, $\kappa>\frac{1}{2}$, and $C_1,C_2>0$ such that
\begin{equation}\label{eq:integrand cond}
\begin{aligned}
        \|\mathbb{E}\left[\delta_u A_{s,t}|\mathcal{F}_s\right]\|_{L^p(\Omega)} &\leq C_1 |t-s|^\beta
        \\
        \|\delta_u A_{s,t}\|_{L^p(\Omega)} &\leq C_2 |t-s|^\kappa.
\end{aligned}
\end{equation}
Then there exists a unique (up to modifications) $(\mathcal{F}_t)$-adapted stochastic process $\mathcal{A}$ such that the following properties are satisfied:
\begin{itemize}[leftmargin=.3in]
    \item[\rm (i)] $\mathcal{A}:[0,T]\rightarrow L^p(\Omega)$ is right continuous, and $\mathcal{A}_0=0$.
    \item[\rm (ii)] There exists two constants $C_1,C_2>0$ such that the following inequalities holds
    \begin{equation}\label{bounds on stochastic integral}
        \begin{aligned}
             \|\mathcal{A}_{s,t}-A_{s,t}\|_{L^p(\Omega)}&\leq C_1 |t-s|^\beta+C_2|t-s|^\kappa
             \\
             \|\EE\left[\mathcal{A}_{s,t}-A_{s,t}|\cF_s\right]\|_{L^p(\Omega)}& \leq C_1 |t-s|^\beta,
        \end{aligned}
    \end{equation}
    where we write $\cA_{s,t}=\cA_t-\cA_s$.
\end{itemize}
Furthermore, for every $(s,t)\in \Delta^2_T$ and partition $\cP$ of $[s,t]$, define
\begin{equation*}
    A^{\cP}_{s,t}:=\sum_{[u,v]\in \cP}A_{u,v}.
\end{equation*}
Then $A^{\cP}_{s,t}$ converge to $\cA_{s,t}$ in $L^p(\Omega)$ when the mesh size $|\cP|\rightarrow 0$.
\end{lem}

\section{Non-linear Young integration and equations}\label{app: nonlinear young equations}
This section is devoted to give the necessary background regarding the non-linear young integral, and Young equations. We begin to prove the existence of the non-linear Young integral.

\begin{lem}\label{lem: non linear young integral}
Let $\Gamma:[0,T]\times \RR^d \rightarrow \RR^d$ be contained in $\cC^\gamma_T\cC^1$  and satisfy the following condition for $x,y\in \RR^d$
\begin{equation}\label{eq: Gamma assumption}
  |\Gamma_{s,t}(x)|\lesssim \|\Gamma\|_{\cC^\gamma_T\cC^1} |t-s|^\gamma \qquad {\rm and} \qquad   |\Gamma_{s,t}(x)-\Gamma_{s,t}(y)|\lesssim \|\Gamma\|_{\cC^\gamma_T\cC^1}|x-y||t-s|^\gamma,
\end{equation}
for some $\gamma\in (0,1)$.
Furthermore suppose $y:[0,T]\rightarrow \RR^d$ is contained in $ \cC^\eta_T$ such that $\gamma+\eta>1$. For a partition $\cP$ of the interval $[0,T]$, define $\varXi_{s,t}:=\Gamma_{s,t}(y_s)$ and the sum
\begin{equation}\label{eq: abstract R sum}
    \cI_\cP = \sum_{[u,v]\in \cP} \varXi_{u,v}.
\end{equation}
Then there exists a unique function $\cI\in \cC^\gamma_T$ satisfying $\cI_{s,t}=\cI_t-\cI_s$ given by $\cI_t:=\lim_{|\cP|\rightarrow 0}\cI_\cP$. We then define
\begin{equation*}
    \int_s^t \Gamma_{\dd r}(y_r)=\cI_{s,t}.
\end{equation*}
Moreover, we have that $\|\delta \Gamma(y)\|_{\cC^{\gamma+\eta}}\lesssim \|\Gamma\|_{\cC^\gamma_T\cC^1}\|y\|_{\cC^\eta}$, and it follows from \cite[Lemma 4.2]{Friz2014} that
\begin{equation}\label{eq: inc of int ineq}
    |\int_s^t \Gamma_{\dd r}(y_r)-\Gamma_{s,t}(y_s)|\lesssim \|\Gamma\|_{\cC^\gamma_T\cC^1}\|y\|_{\cC^\eta}|t-s|^{\eta+\gamma}.
\end{equation}
\end{lem}
\begin{proof}
To prove this, we make use of the classical sewing lemma from the theory of rough paths (see \cite[Lemma 4.2]{Friz2014}). Set $\varXi_{s,t}=\Gamma_{s,t}(y_s)$. Then we know from the sewing lemma that the abstract Riemann sum in \eqref{eq: abstract R sum} converge and \eqref{eq: inc of int ineq} holds  if there exists a $\beta>1$, such that
\begin{equation}\label{eq: delta varxi bound}
   |\varXi_{s,t}|\lesssim |t-s|^\gamma \qquad  |\delta_u \varXi_{s,t}|\lesssim |t-s|^\beta,
\end{equation}
where $\delta_u\varXi_{s,t}=\varXi_{s,t}-\varXi_{s,u}-\varXi_{u,t}$.  It follows by elementary algebraic manipulations that
\begin{equation}\label{eq: algebraic manipulation}
    \delta_u\varXi_{s,t}=\Gamma_{u,t}(y_s)-\Gamma_{u,t}(y_u),
\end{equation}
and thus invoking the assumption in \eqref{eq: Gamma assumption}, we obtain that
\begin{equation*}
    |\delta_u\varXi_{s,t}|\lesssim |t-u|^\gamma|y_s-y_u|\lesssim |t-s|^{\gamma+\eta}.
\end{equation*}
Since $\Gamma\in \cC^\gamma_T$, and due to the assumption that $\eta+\gamma>1$, it follows that \eqref{eq: delta varxi bound} is satisfied, and thus our claim follows from  the sewing lemma (\cite[Lemma 4.2]{Friz2014}).
\end{proof}
\begin{rem}
Of course, the non-linear Young integral coincides with the classical Young integral if the abstract integrand $\Gamma_{s,t}(y_s)$ is for example given by  $\Gamma_{s,t}(y_s)=y_s X_{s,t}$, for some $\gamma$-H\"older continuous path $x$. Furthermore, if $b$ is a measurable function, and $z$ is path of finite $p$-variation, then set $\Gamma_{s,t}(y_s)=\int_s^tb(y_s+z_r)\dd r$. In this case it is readily checked that the integral coincides with the classical Riemann integral
\begin{equation*}
    \int_0^t b(y_r+z_r)\dd r=\int_0^t \Gamma_{\dd r}(y_r).
\end{equation*}
See \cite{galeati2020noiseless} for a comprehensive introduction and discussion of the non-linear integral.
\end{rem}

For self-containedness we include a proof of Lemma \ref{lem: abstract young equations}. The existence and uniqueness of these equations has been proven in \cite{Catellier2016,HarangPerkowski2020}, and we refer to these references for a full account on these results.

\begin{proof}[Proof of Lemma \ref{lem: abstract young equations}]
This proof follow along the lines of \cite[Lemma 30]{HarangPerkowski2020}, and thus we only give here a shorter recollection of the most important details. Let $\beta\in (\frac{1}{\kappa},\gamma)$ where we recall that $\kappa>\frac{1}{\gamma}$ by assumption,   and let $\fS_T:\cC^\beta_T(\RR^d)\rightarrow \cC^\beta_T(\RR^d)$ be the solution map given  by
\begin{equation*}
    \fS_T(y):=\bigg\{\xi+\int_0^t\Gamma_{\dd r}(y_r)\big|\, t\in [0,T]\bigg\}.
\end{equation*}
Let $\scB_{T}(\xi)\subset \cC^\beta_T(\RR^d)$ be a unit ball centered at $\xi\in\RR^d$. In order to prove existence and uniqueness of \eqref{eq: general ODE}, we will begin to show that there exists a $\tau>0$ such that $\fS_\tau$ leaves the unit ball $\scB_\tau(\xi)$ invariant. In the second step we will show that there exists a $\tau^\prime>0$ such that the solution map $\fS_{\tau^\prime}$ is a contraction on the unit ball $\scB_{\tau^\prime}(\xi)$. It then follows by Picard's fixed point theorem that a unique solution  exists in the unit ball $\scB_{\bar{\tau}}(\xi)$ for $\bar{\tau}=\tau\wedge \tau^\prime$. In the end, since $\xi\mapsto \Gamma(\xi)$ is globally bounded, we can iterate the solution to the intervals $[k\bar{\tau},(k+1)\bar{\tau}\wedge T]$ for $k\in \NN$.
\\

\noindent
We begin to show the invariance. By application of \eqref{eq: inc of int ineq} and {\rm (i)} in \eqref{eq: conditions for ex and uni}, it follows that for $y\in \scB_\tau(\xi)$
\begin{equation*}
    \|\fS_\tau(y)\|_{\cC^\beta_\tau} \lesssim \|\Gamma\|_{\cC^\gamma_\tau L^\infty}\tau^{\gamma-\beta}+\|\Gamma\|_{\cC^\gamma_\tau\cC^{\kappa}}\|y\|_{\cC^\beta_\tau}\tau^{\gamma}.
\end{equation*}
Using that $y\in \scB_\tau(\xi)$, and thus in particular $\|y\|_{\cC^\beta_\tau}\leq 1$, it follows that
\begin{equation*}
     \|\fS_\tau(y)\|_{\cC^\beta_\tau} \lesssim \|\Gamma\|_{\cC^\gamma_\tau\cC^{\kappa}} \tau^{\gamma-\beta}.
\end{equation*}
By choosing $\tau>0$ sufficiently small, we obtain $\|\fS_\tau(y)\|_{\cC^\beta_\tau} \leq 1$, and thus $\fS_\tau$ leaves the unit ball $\scB_\tau(\xi)$ invariant.
We continue to prove the contraction property. Applying \eqref{eq: inc of int ineq} it follows from Lemma \ref{lem: non linear young integral}  that for $y,z\in \scB_{\tau^\prime}(\xi)$ we have
\begin{equation}\label{eq: contr ineq}
\begin{aligned}
    \|\fS_{\tau^\prime}(y)-\fS_{\tau^\prime}(z)\|_{\cC^\beta_{\tau^\prime}} &\lesssim \|\Gamma(y)-\Gamma(z)\|_{\cC^\beta_{\tau^\prime}}+ \|\int_0^\cdot( \Gamma_{\dd r}(y)-\Gamma_{\dd r}(z))-(\Gamma(y)-\Gamma(z))\|_{\cC^\beta_{\tau^\prime}}
    \\
    &\lesssim
    \|\Gamma(y)-\Gamma(z)\|_{\cC^\beta_{\tau^\prime}}+\|\delta [\Gamma(y)-\Gamma(z)]\|_{\cC^{\beta^\prime}}(\tau^\prime)^{\beta^\prime-\beta},
    \end{aligned}
\end{equation}
for some $\beta^\prime>1$.
We may assume that $z_0=y_0=\xi$.
For the first term on the right hand side, it follows by assumption \eqref{eq: Gamma assumption} {\rm (ii)} that
\begin{equation}\label{eq:first bound diff}
    \|\Gamma(y)-\Gamma(z)\|_{\cC^\beta_{\tau^\prime}}\lesssim_\Gamma (\tau^\prime)^{\gamma-\beta}\|y-z\|_{\cC^\beta_{\tau^\prime}},
\end{equation}
where we have used that $y_0-z_0=0$. For the second term on the right hand side of \eqref{eq: contr ineq}  we appeal to the proof of the non-linear Young integral in Lemma \ref{lem: non linear young integral}, we will need to show that the action of the $\delta$-operator on the integrand $\varXi_{s,t}:=\Gamma_{s,t}(y_s)-\Gamma_{s,t}(z_s)$ is sufficiently regular and has a contractive property. That is,  we will prove that for $(s,u,t)\in \Delta^3_{\tau^\prime}$, the following inequality holds  $|\delta_u\varXi_{s,t}|\lesssim |t-s|^\mu \|y-z\|_{\cC^\beta}$.
By the fundamental theorem of calculus, it follows that
\begin{equation*}
    \varXi_{s,t}=\int_0^1 \nabla \Gamma_{s,t}(\rho y_s + (1-\rho)z_s)d\rho (y_s-z_s).
\end{equation*}
By the same algebraic manipulations as used in \eqref{eq: algebraic manipulation}, it is readily checked that for $(s,u,t)\in \Delta^3_T$ we have
\begin{equation*}
    \delta_u \varXi_{s,t}=\int_0^1  \nabla \Gamma_{u,t}(\rho y_s + (1-\rho)z_s) (y_s-z_s)-\nabla \Gamma_{u,t}(\rho y_u + (1-\rho)z_u)d\rho (y_u-z_u)d\rho.
\end{equation*}
By addition and subtraction of $\nabla \Gamma_{u,t}(\rho y_s + (1-\rho)z_s) (y_u-z_u)$ inside the above integral, invoking {\rm (i)} of  \eqref{eq: conditions for ex and uni} and using that $\kappa\geq 1$, we begin to observe that
\begin{equation}\label{eq:p1}
    |\nabla \Gamma_{u,t}(\rho y_s + (1-\rho)z_s)[ (y_s-z_s)- (y_u-z_u)]|\lesssim \| \Gamma\|_{\cC^\gamma_T\cC^{\kappa}}\|y-z\|_{\cC^\beta_{\tau^\prime}}|t-s|^{\gamma+\beta}.
\end{equation}
Furthermore, invoking {\rm (iii)} of \eqref{eq: conditions for ex and uni}, it follows that
\begin{multline}\label{eq:p2}
    |[\nabla \Gamma_{u,t}(\rho y_s + (1-\rho)z_s)-\nabla \Gamma_{u,t}(\rho y_u + (1-\rho)z_u)](y_u-z_u)|
    \\
    \lesssim [ \|y\|_{\cC^\beta_{\tau^\prime}}\vee\|z\|_{\cC^\beta_{\tau^\prime}} ]^{\kappa-1} \|\Gamma\|_{\cC^\gamma_{T}\cC^\kappa}|t-u|^\gamma |u-s|^{\beta(\kappa-1)}(|y_0-z_0|+\|y-z\|_{\cC^\beta_{\tau^\prime}}).
\end{multline}
Due to the assumption that $\beta\in (\frac{1}{\kappa},\gamma)$ it follows that $\beta(\kappa-1)+\gamma>1$. Combining \eqref{eq:p1} and \eqref{eq:p2}, it follows that for $y,z\in \scB_{\tau^\prime}(\xi)$ with $z_0=y_0=\xi$ we set $\beta^\prime=\beta(\kappa-1)+\gamma$ and we have
\begin{equation}\label{eq:second bound diff}
    \|\delta_u \varXi_{s,t}\|_{\cC^{\beta^\prime}_{\tau^\prime}}\lesssim_{\Gamma} \|y-z\|_{\cC^\beta_{\tau^\prime}}.
\end{equation}
Thus inserting \eqref{eq:first bound diff} and \eqref{eq:second bound diff} into the right hand side of \eqref{eq: contr ineq}, we obtain the inequality
\begin{equation*}
    \|\fS_{\tau^\prime}(y)-\fS_{\tau^\prime}(z)\|_{\cC^\beta_{\tau^\prime}} \lesssim_\Gamma \|y-z\|_{\cC^\beta_{\tau^\prime}}(\tau^\prime)^{\gamma-\beta}.
\end{equation*}
By choosing $\tau^\prime>0$ small enough, it is clear that the solution map $\fS_{\tau^\prime}$ is a contraction on the ball $\scB_{\tau^\prime}(\xi)$. Note in particular that the contraction bound is independent on the initial data, due to the assumption of boundedness of the derivatives of $\Gamma$ (recall that $\cC^\gamma_\tau\simeq C^\gamma_b([0,\tau])$ when $\gamma\in (0,1)$).

It follows that $\fS_{\tau\wedge \tau^\prime}$ is a contraction and leaves the ball $\scB_{\tau^\prime}(\xi)$ invariant, and it follows by Picard's fixed point theorem that there exists a unique solution to \eqref{eq: general ODE} on in $\scB_{\tau^\prime}(\xi)$. By standard procedures, one can now iterate the solution to the whole interval $[0,T]$, and we ask the patient reader to consult \cite[Section 8.3]{Friz2014} for further details on this part.
At last we note that the solution is indeed contained in the space $\cC^\gamma_T$. Indeed, assume $y\in \cC^\beta_T$ satisfies \eqref{eq: general ODE}. Using the inequality in \eqref{eq: delta varxi bound} the following inequality holds
\begin{equation*}
    |y_{s,t}|=|\int_s^t \Gamma_{\dd r}(y_r)|\lesssim |\Gamma_{s,t}(y_s)|+\|\Gamma\|_{\cC^\gamma_T\cC^\kappa}\|y\|_{\cC^\beta_T}|t-s|^{\gamma+\beta} \lesssim_{y,\Gamma,T} |t-s|^\gamma,
\end{equation*}
and it follows that $y\in \cC^\gamma_T$. This concludes our proof.

\end{proof}

\bibliographystyle{plain}
\bibliography{all}

\end{document}